\documentclass{amsart}[12pt]
\usepackage{amssymb,amsmath}
\usepackage{enumerate}
\usepackage[english]{babel}
\usepackage{color}
\usepackage{comment}

\newcommand {\ep} {\varepsilon}

\newcommand {\gm} {\gamma}

\newcommand {\dt} {\delta}
\newcommand {\al} {\alpha}

\newcommand {\su} {\subset}

\newcommand {\wh} {\widehat}

\newcommand {\mc} {\mathcal}

\newtheorem{teo}{Theorem}[section]
\newtheorem{pro}{Proposition}[section]
\newtheorem{cor}{Corollary}[section]

\newtheorem{lm}{Lemma}[section]

\theoremstyle{definition}
\newtheorem{rem}{Remark}[section]
\newtheorem{df}{Definition}[section]

\newtheorem{ex}{Example}[section]

\title[Noncommutative Wiener-Wintner type ergodic theorems]{Noncommutative Wiener-Wintner type ergodic theorems}

\author{Morgan O'Brien}
\address{North Dakota State University\\ Department of Mathematics\\ 1210 Albrecht Boulevard, Minard Hall \\ Fargo, ND 58102, USA}
\email{morgan.obrien@ndsu.edu, obrienmorganc@gmail.com}

\begin{document}
\begin{abstract}
In this article, we obtain a version of the noncommutative Banach Principle suitable to prove Wiener-Wintner type results for weights in $W_1$-space. This is used to obtain noncommutative Wiener-Wintner type ergodic theorems for various types of weights for certain types of positive Dunford-Schwartz operators. We also study the b.a.u. (a.u.) convergence of some subsequential averages and moving averages of such operators.
\end{abstract}
\date{November 23, 2021}

\maketitle
\section{Introduction}\label{s1}

Weighted (or modulated) ergodic theorems have been an active area of research in ergodic theory since the early 1980's \cite{bl,comlio,ei,el,lot}. Results of this nature have recently been of interest in the von Neumann algebra setting as well (see, for example, \cite{cl1,cl2,cls,hs,li2,lm,ob}).

Loosely speaking, given a measure space $(X,\mathcal{F},\mu)$, a function $f\in L_1(X,\mathcal{F},\mu)$, and an operator $T$ acting on $L_1(X,\mathcal{F},\mu)$, if a weighted pointwise ergodic theorem holds for a complex sequence $\alpha=\{\alpha_k\}_{k=0}^{\infty}$, then one may find a set $X_\alpha\subseteq X$ of full measure for which the weighted averages $\frac{1}{n}\sum_{k=0}^{n-1}\alpha_k T^kf(x)$ converge for every $x\in X_\alpha$. If, furthermore, all sequences in some collection $\mathcal{W}$ are to be considered, one would probably use $\bigcap_{\alpha\in\mathcal{W}}X_\alpha$ as the set of convergence. However, this intersection may have a small measure or be empty, even if the weighted pointwise ergodic theorem holds for every $\alpha\in\mathcal{W}$ (for example, it may be the case that, for every $x\in X$, there is some $\alpha\in\mathcal{W}$ such that $x\not\in X_\alpha$).

Such a situation does not occur if a Wiener-Wintner type ergodic theorem holds for $\mathcal{W}$. That is, given a function $f\in L_1(X,\mathcal{F},\mu)$ and operator $T$ acting on $L_1(X,\mathcal{F},\mu)$, there is a set $X_{\mathcal{W}}$ of full measure in $X$ on which the weighted ergodic theorem holds for every $\alpha\in\mathcal{W}$ by taking $X_\alpha=X_{\mathcal{W}}$.


The class of Hartman (almost periodic) sequences - which arise naturally in harmonic analysis, ergodic theory, and number theory as sequences that admit Fourier coefficients - is one of the most common classes of sequences for weighted ergodic theorems. 
A complex sequence $\alpha$ is a Hartman sequence when $c_{\alpha}(\lambda):=\lim_{n\to\infty}\frac{1}{n}\sum_{k=0}^{n-1}\alpha_k\overline{\lambda}^k$ exists for every $\lambda\in\mathbb{T}=\{\lambda\in\mathbb{C}:|\lambda|=1\}$. For bounded sequences, it is known that being a Hartman sequence is equivalent to being a good weight for the mean ergodic theorem for every contraction on a Hilbert space \cite[Theorem 21.2]{efhn}, while Hartman sequences which are not necessarily bounded still frequently occur as good weights for the same types of result \cite{comlio,el,lot}.

For any probability measure-preserving (p.m.p.) dynamical system $(X,\mathcal{F},\mu,\phi)$, the Wiener-Wintner theorem states that, for any $f\in L_1(X,\mathcal{F},\mu)$, there is a measurable set $X'\subseteq X$ with $\mu(X')=1$ such that the weighted averages $\frac{1}{n}\sum_{k=0}^{n-1}\lambda^k f(\phi^k(x))$ converge for every $x\in X'$ and every $\lambda\in\mathbb{T}$. In other words, every $x\in X'$ gives rise to a Hartman sequence $\{f(\phi^k(x))\}_{k=0}^{\infty}$. A more thorough discussion of this can be found in \cite{efhn,lot}.

Another important class of Hartman sequences is the class of Besicovich sequences, which have been widely studied in classical ergodic theory (see \cite{bl,lot}), and have also been studied in the context of individual ergodic theorems on von Neumann algebras in \cite{cl1,cl2,cls}. Other types of Hartman sequences were studied as weights in the noncommutative setting in \cite{hs}, where a noncommutative multiparameter Wiener-Wintner type ergodic theorem was proven.

In the commutative case, almost everywhere (a.e.) convergence of weighted averages of a power-bounded operator with unbounded weights is known in numerous situations (see, for example, \cite{comlio,lot}). On the other hand, in the noncommutative setting, all known results of this type are obtained for bounded weights. In this article, we prove, among others, a noncommutative Wiener-Winter type ergodic theorem for weights which are not necessarily bounded.

As stated in \cite[Section 3]{lot}, there exists a bounded Hartman sequence $\alpha=\{\alpha_k\}_{k=0}^{\infty}$, a p.m.p. system $(X,\mathcal{F},\mu,\phi)$, and $f\in L_\infty(X,\mathcal{F},\mu)$ such that the weighted averages $\frac{1}{n}\sum_{k=0}^{n-1}\alpha_k \, f\circ\phi^k$ fail to converge $\mu$-a.e. Note that the map $f\mapsto f\circ\phi$ is a positive Dunford-Schwartz operator on $L_\infty(X,\mathcal{F},\mu)$. 
Consequently, in the noncommutative setting, if we want to obtain almost uniform (a.u.) or bilaterally almost uniform (b.a.u.) convergence of the averages of a positive Dunford-Schwartz operator $T$ weighted by sequences in a set which includes all bounded Hartman sequences, some extra assumptions on $T$ are needed.

In \cite[Corollary 1]{st}, it was shown that, if $(X,\mathcal{F},\mu)$ is a probability space, and if $T$ is a Dunford-Schwartz operator on $L_\infty(X,\mathcal{F},\mu)$ such that the restriction of $T$ to $L_2(X,\mathcal{F},\mu)$ is positive as a bounded operator on the Hilbert space, then for every $f\in L_p(X,\mathcal{F},\mu)$ with $1<p<\infty$, the iterates $T^nf(x)$ converge a.e. on $X$. This was extended to the von Neumann algebra setting in \cite[Theorem 6.7]{jx} by considering a.u. and b.a.u. convergence for a positive Dunford-Schwartz operator $T$ such that the restriction of $T$ to $L_2(\mathcal{M},\tau)$ is a positive Hilbert space operator. This result was then generalized to hold for $T$ satisfying weaker conditions in \cite[Theorem 4.3]{be}.

Let $\mathcal{M}$ be a von Neumann algebra equipped with a normal semifinite faithful trace $\tau$, and let $T$ be a weakly almost periodic operator (WAP) on the Banach space $L_p(\mathcal{M},\tau)$, $1<p<\infty$. The Jacobs-de Leeuw-Glicksberg decomposition $L_p(\mathcal{M},\tau)=\overline{\operatorname{span}(\mathcal{U}_p(T))}\oplus\mathcal{V}_p(T)$ \cite[Chapter 16]{efhn} gives us a tool commonly used to prove various ergodic theorems. If $T$ is a positive Dunford-Schwartz operator like above, then for every $x\in\mathcal{V}_p(T)$ it follows that $T^n(x)\to0$ b.a.u. or a.u. as $n\to\infty$ according to the value of $p$. With this in mind, our remedy to the problem mentioned above will be to assume that $T^n(x)\to0$ b.a.u. or a.u. only for elements in $\mathcal{V}_p(T)$. It will be shown in Section \ref{s6} that this assumption is more natural than it appears. For instance, if the restriction of $T$ to $L_2(\mathcal{M},\tau)$ is only self-adjoint, then this assumption is satisfied. We note that it is possible that $T(y)=-y$ for some nonzero $y\in L_p(\mathcal{M},\tau)$, so that $\{T^n(y)\}_{n=0}^{\infty}$ does not converge b.a.u. at all. 

The layout of the paper, and results obtained, are as follows. In Section \ref{s3}, we study the properties of closedness of the set of operators for which a noncommutative version of the Wiener-Wintner type ergodic theorem holds. In particular, Theorem \ref{t31} states that, given a set $\mathcal{W}\subseteq W_1$, where $W_1$ denotes the set of sequences which are absolutely Ces\'{a}ro bounded, and certain conditions on a positive Dunford-Schwartz operator $T$, the set of operators in $L_p(\mathcal{M},\tau)$ for which the noncommutative counterpart of the Wiener-Wintner property holds for $\mathcal{W}$ is a closed subset of $L_p(\mathcal{M},\tau)$ for $1\leq p<\infty$. Furthermore, in Proposition \ref{p33}, under a certain equicontinuity condition on the iterates of $T$, we prove that one can extend the weights under consideration in Theorem \ref{t31} to their closure with respect to the topology on $W_1$.

In Section \ref{s4}, we utilize the assumption on the b.a.u. (a.u.) convergence of the iterates of $T$ for $x\in\mathcal{V}_q(T)$ for $1<q<\infty$ (respectively, $2\leq q<\infty$) to obtain a noncommutative Wiener-Wintner ergodic theorem for the set of all bounded Hartman sequences. This result is then extended to $L_p(\mathcal{M},\tau)$ for every other $1\leq p<\infty$ (respectively, $2\leq p<\infty$). If we further require that the sequence of iterates $\{T^n\}_{n=0}^{\infty}$ are bilaterally uniformly (respectively, uniformly) equicontinuous in measure at zero on $(L_q(\mathcal{M},\tau),\|\cdot\|_q)$, this result will be shown to hold on $L_q(\mathcal{M},\tau)$ when the weights are Hartman sequences which are in the $W_1$-seminorm closure of the union of the $W_r$ class of weights over $1<r\leq\infty$. As corollaries to these results, we obtain several Return Time-type theorems for the operators involved.

We study the convergence of subsequential and moving averages of such operators in Section \ref{s5}. In particular, we show that b.a.u. and a.u. convergence occurs when the powers of $T$ are prime numbers. This result, and the results in \cite{el}, imply that the von Mangoldt function provides a good weight for $T$ as well. We then use the norm convergence results of \cite{comli} to obtain a Wiener-Wintner type ergodic theorem for moving average sequences.

Finally, we discuss numerous examples and conditions that show the assumptions on $T$ are actually attainable. Notably, as mentioned above, we show in Example \ref{e65} that, if the restriction of $T$ to $L_2(\mathcal{M},\tau)$ is only self-adjoint, then $T$ satisfies the assumptions made throughout the article.

\section{Preliminaries}\label{s2}

Throughout this article, we will let $\mathcal{M}$ be a semifinite von Neumann algebra acting on a Hilbert space $\mathcal{H}$. Further, we will let $\tau$ denote a normal semifinite faithful trace on $\mathcal{M}$. The operator norm on $\mathcal{B}(\mathcal{H})$ will be denoted by $\|\cdot\|_\infty$. Denote the identity operator on $\mathcal{H}$ by $\textbf{1}$. For clarity, $\mathbb{N}$ will denote the set of natural numbers, $\mathbb{N}_0=\mathbb{N}\cup\{0\}$, and $\mathbb{T}\subset\mathbb{C}$ will be the unit circle in the set of complex numbers.

We will denote the lattice of projections on $\mathcal{M}$ by $\mathcal{P}(\mathcal{M})$. Given $e\in\mathcal{P}(\mathcal{M})$, write $e^\perp:=\textbf{1}-e$, noting that $e^\perp\in\mathcal{P}(\mathcal{M})$ as well.

Suppose $x:\mathcal{D}_x\to\mathcal{H}$ is a closed densely defined operator on $\mathcal{H}$. Then $x$ is \textit{affiliated} to $\mathcal{M}$ if $yx\subseteq xy$ for every $y\in\mathcal{M}'$, the commutant of $\mathcal{M}$. If $x$ is affiliated to $\mathcal{M}$, then $x$ is $\tau$-measurable if for every $\epsilon>0$ there exists $e\in\mathcal{P}(\mathcal{M})$ such that $\tau(e^\perp)\leq\epsilon$ and $xe\in\mathcal{M}$. Let $L_0(\mathcal{M},\tau)$ denote the set of all $\tau$-measurable operators on $\mathcal{H}$. Define, for every $\epsilon,\delta>0$, the set $V(\epsilon,\delta)$ by
$$V(\epsilon,\delta)=\{x\in L_0(\mathcal{M},\tau):\|xe\|_\infty\leq\delta\text{ for some }e\in\mathcal{P}(\mathcal{M})\text{ with }\tau(e^\perp)\leq\epsilon\}.$$ Then $\{V(\epsilon,\delta)\}_{\epsilon,\delta>0}$ is a system of neighborhoods of zero in $L_0(\mathcal{M},\tau)$ which induces a topology, called the \textit{measure topology}, under which it is a complete topological metrizable $*$-algebra (see \cite{ne}).

Given $x\in L_0(\mathcal{M},\tau)$, we say that $x$ is positive if $\langle x\xi,\xi\rangle\geq0$ for every $\xi\in\mathcal{D}_x$, and write $x\geq0$. If $S\subseteq L_0(\mathcal{M},\tau)$, then write $S^+=\{x\in S:x\geq0\}$.

For every $x\in L_0(\mathcal{M},\tau)$, given $t>0$ one may define the $t$\textit{-th generalized singular number of} $x$ as
$$\mu_t(x)=\inf\{\|xe\|_\infty:\text{there exists }e\in\mathcal{P}(\mathcal{M})\text{ such that }\tau(e^\perp)\leq t\}.$$ Using this, one may extend the trace $\tau$ to all of $L_0^+(\mathcal{M},\tau)$. Indeed, by \cite[Proposition 2.7]{fk}, given $x\in L_0^+(\mathcal{M},\tau)$, one has
$$\tau(x)=\int_{0}^{\infty}\mu_t(x)dt.$$

For each $x\in L_0(\mathcal{M},\tau)$, there exists $u\in\mathcal{M}$ and $|x|\in L_0^+(\mathcal{M},\tau)$ such that $x=u|x|$, where $|x|$ is given by $|x|=(x^*x)^{1/2}$. For each $1\leq p<\infty$, we define  $L_p(\mathcal{M},\tau)$, the noncommutative $L_p$-space associated to $\mathcal{M}$, by
$$L_p(\mathcal{M},\tau):=\{x\in L_0(\mathcal{M},\tau):\tau(|x|^p)<\infty\}.$$
Given $x\in L_p(\mathcal{M},\tau)$, define $\|x\|_p=(\tau(|x|^p))^{1/p}$. Write $L_\infty(\mathcal{M},\tau):=\mathcal{M}$, and equip it with the operator norm $\|\cdot\|_\infty$. Then $(L_p(\mathcal{M},\tau),\|\cdot\|_p)$ is a Banach space for every $1\leq p\leq\infty$. For sake of brevity, for either $p=0$ or $1\leq p\leq\infty$, we will frequently write $L_p$ in place of $L_p(\mathcal{M},\tau)$.

A Dunford-Schwartz operator is a linear operator $T:L_1+\mathcal{M}\to L_1+\mathcal{M}$ with
$$\|T(x)\|_1\leq\|x\|_1 \ \forall x\in L_1 \ \text{ and } \ \|T(x)\|_\infty\leq\|x\|_\infty \ \forall x\in\mathcal{M}.$$
If $T(x)\geq0$ for every $x\geq0$, then $T$ is called a positive Dunford-Schwartz operator, and we write $T\in DS^+:=DS^+(\mathcal{M},\tau)$. Notably, by \cite[Lemma 1.1]{jx}, $T$ extends to a linear contraction on $L_p$ for every $1\leq p\leq\infty$.

Throughout this article we will consider the weighted averages of the iterates of $T$. Fix $\alpha=\{\alpha_k\}_{k=0}^{\infty}\subset\mathbb{C}$. Then, for every $n\in\mathbb{N}$, define the $n$\textit{-th weighted ergodic average of $T$ by $\alpha$} as
$$M_n^\alpha(T)(x)=\frac{1}{n}\sum_{k=0}^{n-1}\alpha_kT^k(x), \text{ where } x\in L_1+\mathcal{M}.$$
When $\alpha=\{1\}_{n=0}^{\infty}$, we will write $M_n^\alpha(T)=M_n(T)$ for every $n\geq1$.

A sequence $\{x_n\}_{n=1}^{\infty}\subset L_0$ is said to converge to $x\in L_0$ \textit{bilaterally almost uniformly (b.a.u.) (almost uniformly (a.u.))} if, for every $\epsilon>0$, there exists $e\in\mathcal{P}(\mathcal{M})$ such that $\tau(e^\perp)\leq\epsilon$ and
$$\lim_{n\to\infty}\|e(x_n-x)e\|_\infty=0 \ (\text{respectively,  }\lim_{n\to\infty}\|(x_n-x)e\|_\infty=0).$$

If $\mathcal{M}$ is commutative, then a.u. and b.a.u. convergence are the same;
however, in general, a.u. convergence implies b.a.u. convergence (see \cite[Example 3.1]{cl2} for the failure of the converse).

\begin{df}\label{d22}
Let $(X,\|\cdot\|)$ be a normed space. Let $A_n:X\to L_0$, $n\in\mathbb{N}_0$, be a sequence of linear maps. Then $\{A_n\}_{n=0}^{\infty}$ is said to be \textit{bilaterally uniformly equicontinuous in measure (b.u.e.m.) (uniformly equicontinuous in measure (u.e.m.)) at zero on $(X,\|\cdot\|)$} if, for every $\epsilon,\delta>0$, there exists $\gamma>0$ such that, for every $x\in X$ with $\|x\|<\gamma$, there exists $e\in\mathcal{P}(\mathcal{M})$ such that $\tau(e^\perp)\leq\epsilon$, $A_n(x)e\in\mathcal{M}$, and
$$\sup_{n\in\mathbb{N}_0}\|eA_n(x)e\|\leq\delta \ (\text{respectively, }\sup_{n\in\mathbb{N}_0}\|A_n(x)e\|\leq\delta).$$
\end{df}

\begin{rem}\label{r21}
If linear maps $A_n: X\to L_0$, $n=0,1,\dots$, are continuous with respect to the measure topology in $L_0$, 
then b.a.u (a.u.) convergence of the sequence $\{A_n(x)\}_{n=1}^{\infty}$ for each $x\in X$ implies that $\{A_n\}_{n=1}^{\infty}$ is b.u.e.m. (respectively, u.e.m.) at zero on $X$ by the results in \cite{cl3}. This is guaranteed if there exists $1\leq p<\infty$ such that the ranges of these maps lie in $L_p$ and they are continuous with respect to $\|\cdot\|_p$. 

For example, if $T\in DS^+$ is such that the sequence $\{T^n(x)\}_{n=0}^{\infty}$ converges b.a.u. (a.u.) for each $x\in L_p$ with $1\leq p<\infty$, then the sequence $\{T^n\}_{n=0}^{\infty}$ is b.u.e.m (respectively, u.e.m.) at zero on $(L_p,\|\cdot\|_p)$.
\end{rem}

\begin{pro}\cite[Theorem 4.4, Proposition 4.3]{li1}\label{p21}
Assume $T\in DS^+(\mathcal{M},\tau)$. Then, for any $1\leq p<\infty$, the sequence $\{M_n(T)\}_{n=1}^{\infty}$ is b.u.e.m. at zero on $(L_p,\|\cdot\|_p)$. Furthermore, if $2\leq p<\infty$, then $\{M_n(T)\}_{n=1}^{\infty}$ are u.e.m. at zero on $(L_p,\|\cdot\|_p)$.
\end{pro}

\begin{pro}\cite[Theorem 2.1]{li1}, \cite[Theorem 2.3]{cl3}\label{p22}
Let $X$ be a Banach space, and let $A_n:X\to L_0$ be a sequence of linear maps. If $\{A_n\}_{n=0}^{\infty}$ is b.u.e.m. (u.e.m.) at zero on $X$, then the set
$$\mathcal{C}=\{x\in X:\{A_n(x)\}_{n=0}^{\infty} \ \text{ converges b.a.u. (respectively, a.u.)}\}$$ is a closed subspace of $X$. 
\end{pro}

\section{Wiener-Wintner form of the Banach Principle}\label{s3}

In the sequel, whenever a result involves b.a.u. and a.u. versions, we will prove only the b.a.u. case when the proof of the other follows the same reasoning. The same will be the case if the statement involves b.u.e.m. and u.e.m. versions or $bWW_q(\cdot)$ and $WW_q(\cdot)$ versions.

For $1\leq r<\infty$, let $W_r$ denote the set of sequences $\alpha=\{\alpha_k\}_{k=0}^{\infty}\subset\mathbb{C}$ with
\[
\|\alpha\|_{W_r}:=\left(\limsup_{n\to\infty}\frac{1}{n}\sum_{k=0}^{n-1}|\alpha_k|^r\right)^{1/r}<\infty,
\] 
and let $W_\infty$ denote the set of 
bounded sequences in $\mathbb C$ with $\|\alpha\|_{W_\infty}=\sup_{n\in\mathbb{N}_0}|\alpha_n|$, $\alpha\in W_\infty$.
The function $\|\cdot\|_{W_r}$ is called the {\it $W_r$-seminorm}. It is known that each space $W_r$ is complete with respect to the topology generated by $\|\cdot\|_{W_r}$ for $1\leq r\leq\infty$, and that $W_r\subset W_s$ whenever $1\leq s<r\leq\infty$. Let $W_{1^+}$ denote the $W_1$-seminorm closure of $\bigcup_{r>1}W_r$.

For notational convenience, we will write 
\[
|\alpha|_{W_r}:=\left(\sup_n\frac{1}{n}\sum_{k=0}^{n-1}|\alpha_k|^r\right)^{1/r}\Big(\|\alpha\|_{W_\infty}\Big) \text{\ \ if \ }\al=\{\alpha_k\}_{k=0}^{\infty}\in W_r
\]
and $1\leq r<\infty$ (respectively, $r=\infty$).
Due to the fact that $\limsup_{n\to\infty}\alpha_n<\infty$ if and only if $\sup_{n\in\mathbb{N}_0}\alpha_n<\infty$ for every $\{\alpha_n\}_{n=0}^{\infty}\subset\mathbb{R}$, it follows that $|\alpha|_{W_r}<\infty$ whenever $\alpha\in W_r$.

\begin{df}\label{d31}
 Let $T\in DS^+$, $1\leq p<\infty$, $1\leq r\leq\infty$, and $\mathcal{W}\subseteq W_r$. The family $\{M_n^\alpha(T)\}_{\substack{n\in\mathbb N,\\ \alpha\in\mathcal{W}}}$ is called {\it $\mathcal{W}$-b.u.e.m.} ({\it $\mathcal{W}$-u.e.m.}) {\it at zero on $L_p$} if there exists an increasing function $h: [0,\infty)\to[0,\infty)$ such that, given $\ep>0$, $\dt>0$, there is $\gm>0$ for which $\|x\|_p<\gm$ entails the existence of $e\in\mathcal{P}(\mathcal{M})$ with $\tau(e^\perp)\leq\ep$ satisfying
\[
\sup_n\|eM_n^\alpha(T)(x)e\|_\infty\leq h(|\alpha|_{W_r})\delta\text{\, \Big(resp.,\ } \sup_n\|M_n^\alpha(T)(x)e\|_\infty\leq h(|\al|_{W_r})\delta\Big) 
\]
for all $\alpha\in\mathcal{W}$.
\end{df}

Let $\mathcal{W}\subseteq W_1$. Given $1\leq p<\infty$, we will write
\begin{align*}
bWW_p(\mathcal{W})
&=\Big\{x\in L_p:\ \forall\ \epsilon>0 \ \exists\ e\in\mathcal{P}(\mathcal{M})\text{ such that }\tau(e^\perp)\leq\epsilon\text{ and } \\
&\ \ \ \ \ \  \ \ \ \ \ \ \ \ \ \ \ \ \big\{eM_n^\alpha(T)(x)e\big\}_{n=1}^{\infty}\text{ converges in } \mathcal{M}\, \ \forall \ \alpha\in\mathcal{W}\Big\} \\
WW_p(\mathcal{W})
&=\Big\{x\in L_p:\ \forall\ \epsilon>0 \ \exists\ e\in\mathcal{P}(\mathcal{M})\text{ such that }\tau(e^\perp)\leq\epsilon\text{ and } \\
&\ \ \ \ \ \  \ \ \ \ \ \ \ \ \ \ \ \ \big\{M_n^\alpha(T)(x)e\big\}_{n=1}^{\infty}\text{ converges in } \mathcal{M}\, \ \forall \ \alpha\in\mathcal{W}\Big\}.
\end{align*}

The following was proven for a fixed sequence $\beta\in W_\infty$ in \cite{cl1} as Theorems 2.1 and 2.3. However, closer examination of the proofs show that the projection $e\in\mathcal{P}(\mathcal{M})$ obtained does not depend on the particular sequence $\beta$, and so the result may be stated in the following stronger form.

\begin{pro}[cf. \cite{cl1}]\label{p31}
Let $T\in DS^+$. Then $\{M_n^\alpha(T)\}_{n,\alpha}$ is $W_\infty$-b.u.e.m. at zero on $L_p$ with $h(s)=s$ when $1\leq p<\infty$  and $W_\infty$-u.e.m. at zero on $L_p$ with $h(s)=\sqrt{s}(2+\sqrt{s})$ when $2\leq p<\infty$.
\end{pro}

We now want to prove an extension of this result for weights in $W_1$, but will need an extra assumption on the iterates of the operator $T\in DS^+$.

\begin{pro}\label{p32}
Let $1\leq p<\infty$, and assume $T\in DS^+$ is such that $\{T^n\}_{n=0}^\infty$ is b.u.e.m. (u.e.m.) at zero on $(L_p,\|\cdot\|_p)$. Then $\{M_n^\alpha(T)\}_{n,\alpha}$ is $W_1$-b.u.e.m. (respectively, $W_1$-u.e.m.) at zero on $L_p$.
\end{pro}
\begin{proof}
We will prove the b.u.e.m. at zero case only. 

Fix $\epsilon>0$ and $\delta>0$, and let $\gamma>0$ be such that $\|x\|_p<\gamma$ implies that there exists $e\in\mathcal{P}(\mathcal{M})$ with $\tau(e^\perp)\leq\epsilon$ satisfying $\sup_n\|eT^n(x)e\|_\infty\leq\delta$. Assume $x\in L_p$ is such that $\|x\|_p<\gamma$, and let $e\in\mathcal{P}(\mathcal{M})$ be such that $\tau(e^\perp)\leq\epsilon$ and $\sup_k\|eT^k(x)e\|_\infty\leq\delta.$ Then, given $\alpha\in W_1$, we have
\begin{align*}
\|eM_n^\alpha(T)(x)e\|_\infty
&\leq\frac{1}{n}\sum_{k=0}^{n-1}|\alpha_k|\|eT^k(x)e\|_\infty \\ 
&\leq\left(\frac{1}{n}\sum_{k=0}^{n-1}|\alpha_k|\right)\sup_k\|eT^k(x)e\|_\infty\leq|\alpha|_{W_1}\delta.
\end{align*}
for every $n$, 
so the claim follows with $h(s)=s$.
\end{proof}

The statement below is a Wiener-Wintner form of the noncommutative Banach Principle for weights in $W_1$:
\begin{teo}\label{t31}
Let $T\in DS^+$, $1\leq p<\infty$, $1\leq r\leq\infty$, and let $\mathcal{W}\subseteq W_r$. If $\{M_n^\alpha(T)\}_{n,\alpha}$ is $\mathcal{W}$-b.u.e.m. ($\mathcal{W}$-u.e.m.) at zero on $L_p$, then $bWW_p(\mathcal{W})$ (respectively, $WW_p(\mathcal{W})$) is closed in $L_p$.
\end{teo}
\begin{proof}
We will prove the case of $bWW_p(\mathcal{W})$ only. Without loss of generality, assume that $h(s)=s$ (otherwise replace $|\alpha|_{W_r}$ and $m$ with $h(|\alpha|_{W_r})$ and $h(m)$ whenever needed).

Assume $x\in\overline{bWW_p(\mathcal{W})}$ and $\{z_j\}_{j=1}^{\infty}\su bWW_p(\mathcal{W})$ are such that $\|z_j-x\|_p\to0$ as $j\to\infty$, and let $y_j=z_j-x$ for every $j\geq1$.

Fix $\epsilon>0$, and let $k,m\in\mathbb{N}$. Since $\|y_j\|_p\to0$ as $j\to\infty$, and since $\{M_n^\alpha(T)\}_{n,\alpha}$ is $\mathcal{W}$-b.u.e.m. at zero on $L_p$, we may find $x_{k,m}\in\{y_j\}_j$ and $e_{k,m}\in\mathcal{P}(\mathcal{M})$ such that, for every $\alpha\in\mathcal{W}$, 
\[
\tau(e_{k,m}^{\perp})\leq\frac{\epsilon}{2^{k+m+1}} \ \text{ and } \ \sup_n\|e_{k,m}M_n^\alpha(T)(x_{k,m})e_{k,m}\|_\infty\leq\frac{|\alpha|_{W_r}}{km}.
\]

Now, if $e=\bigwedge_{k,m=1}^{\infty}e_{k,m}$, then it follows that
$\tau(e^\perp)\leq\ep/2$ and, for every $k$, $m$, and $\alpha\in\mathcal{W}$, 
\begin{equation}\label{eq51}
\sup_n\|eM_n^\alpha(T)(x_{k,m})e\|_\infty\leq\sup_n\|e_{k,m}M_n^\alpha(T)(x_{k,m})e_{k,m}\|_\infty
\leq\frac{|\alpha|_{W_r}}{km}.
\end{equation}

Given $k,m\in\mathbb{N}$,  
there exists $j_{k,m}\in\mathbb{N}$ such that $x_{k,m}=y_{j_{k,m}}=z_{j_{k,m}}-x$. As such, $x+x_{k,m}\in bWW_p(\mathcal{W})$ for every $k,m$, so there exists $f_{k,m}\in\mathcal{P}(\mathcal{M})$ such that
\[
\tau(f_{k,m}^{\perp})\leq\frac{\epsilon}{2^{k+m+1}} \ \text{ and } \ \{f_{k,m}M_n^\alpha(T)(x+x_{k,m})f_{k,m}\} \text{ converges in\ } \mathcal{M}\ \ \forall\ \alpha\in\mathcal{W}.
\]
Let $f=\bigwedge_{k,m=1}^{\infty}f_{k,m}$ and note that then $\tau(f^\perp)\leq\ep/2$ 
and, for every $k,m\in\mathbb{N}$,
\begin{equation}\label{eq52}
\{fM_n^\alpha(T)(x+x_{k,m})f\}_{n=1}^{\infty}\text{ converges in } \mathcal{M}\ \ \forall \ \alpha\in\mathcal{W}.
\end{equation}

Next, let $g=e\wedge f$. Then $\tau(g^\perp)\leq\epsilon$, and $g$ satisfies (\ref{eq51}) and (\ref{eq52}) in place of $e$ and $f$ respectively for every $k,m\in\mathbb{N}$.

Fix $\alpha\in\mathcal{W}$, and $\delta>0$. Since $|\alpha|_{W_r}<\infty$, there exists $m_0\in\mathbb{N}$ such that $|\alpha|_{W_r}\leq m_0$. Let $k_0\in\mathbb{N}$ be such that $\frac{1}{k_0}\leq\frac{\delta}{3}$. By (\ref{eq51}), it follows that
\[
\sup_n\|gM_n^\alpha(T)(x_{k_0,m_0})g\|_{\infty}\leq\frac{|\alpha|_{W_r}}{k_0m_0}\leq\frac{1}{k_0}\leq\frac{\delta}{3}.
\]
By (\ref{eq52}) above, $\{gM_n^\alpha(T)(x+x_{k_0,m_0})g\}_{n=1}^{\infty}$ is Cauchy, so that there exists $N$ such that $l,n\geq N$ implies 
\[
\|g(M_l^\alpha(T)(x+x_{k_0,m_0})-M_n^\alpha(T)(x+x_{k_0,m_0}))g\|_\infty\leq\frac{\delta}{3}.
\] 
Therefore, for every $l,n\geq N$, we find that
\begin{align*}
\|g(M_l^\alpha(T)(x)&-M_n^\alpha(T)(x))g\|_\infty\\
&\leq\|g(M_l^\alpha(T)(x+x_{k_0,m_0})-M_n^\alpha(T)(x+x_{k_0,m_0}))g\|_\infty \\
&+ \|gM_l^\alpha(T)(x_{k_0,m_0})g\|_\infty+\|gM_n^\alpha(T)(x_{k_0,m_0})g\|_\infty\leq\delta.
\end{align*}
Therefore, $\{gM_n^\alpha(T)(x)g\}_{n=1}^{\infty}$ is a Cauchy sequence in $\mathcal{M}$,  
hence $x\in bWW_p(\mathcal{W})$, 
so $bWW_p(\mathcal{W})$ is closed in $L_p$.
\end{proof}

As a consequence of this result, we obtain a non-trivial class of weights that works for a Wiener-Wintner-type ergodic theorem. For this, write 
\[
\mathcal{W}_c=\big\{\alpha=\{\alpha_k\}\su\mathbb C:\ \exists \ l(\al)\in\mathbb{C} \text{ \ such that }\lim_{k\to\infty}\alpha_k=l(\al)\big\}
\]
and note that $\mathcal{W}_c\subset W_\infty$.
\begin{cor}\label{c31}
If $T\in DS^+$, then $bWW_p(\mathcal{W}_c)=L_p$ for $1\leq p<\infty$ and $WW_p(\mathcal{W}_c)=L_p$ for $2\leq p<\infty$.
\end{cor}
\begin{proof}
Assume $x\in L_1\cap\mathcal{M}$ and let $\epsilon>0$. Since $x\in L_2$, there exists $\widehat{x}\in L_2$ such that $\{M_n(T)(x)\}_{n=1}^{\infty}$ converges to $\wh x$ a.u. (see, for example, \cite[Theorem 4.1]{li1}). Let $e\in\mathcal{P}(\mathcal{M})$ be such that $\tau(e^\perp)\leq\epsilon$ and $\|(M_n(T)(x)-\widehat{x})e\|_\infty\to0$ as $n\to\infty$.

Assume $\alpha=\{\alpha_k\}_{k=0}^{\infty}\in\mathcal{W}_c$.
Then, since $T$ is a contraction on $\mathcal{M}$,
\begin{align*}
\|(M_n^\alpha(T)(x)-l(\al)\widehat{x})e\|_\infty
&\leq\frac{1}{n}\sum_{k=0}^{n-1}|\alpha_k-l(\al)|\|T^k(x)e\|_\infty\\
&+|l(\al)|\|(M_n(T)(x)-\widehat{x})e\|_\infty \\
&\leq\|x\|_\infty\frac{1}{n}\sum_{k=0}^{n-1}|\alpha_k-l(\al)|+|l(\al)|\|(M_n(T)(x)-\widehat{x})e\|_\infty
\end{align*}
Therefore 
$\|(M_n^\alpha(T)(x)-l(\al)\widehat{x})e\|_\infty\to0$, and so $\{M_n^\alpha(T)(x)e\}$ converges in $\mathcal{M}$. Since $\alpha\in\mathcal{W}_c$ and $\epsilon>0$ were arbitrary, it follows that $x\in WW_p(\mathcal{W}_c)$. Therefore, since $L_1\cap\mathcal{M}$ is dense in $L_p$, it follows, by Proposition \ref{p31} and Theorem \ref{t31}, that $bWW_p(c)=L_p$ with $1\leq p<\infty$ and that $WW_p(c)=L_p$ with $2\leq p<\infty$.
\end{proof}

Next, we discuss a few results that extend the weights allowed in the previous results. In the following results, if $\mathcal{W}\subseteq W_1$, then let $\overline{\mathcal{W}}$ denote the closure of $\mathcal{W}$ in $W_1$ with respect to the $W_1$-seminorm $\|\cdot\|_{W_1}$.

\begin{pro}\label{p33}
Let $1\leq p<\infty$, and assume that $T\in DS^+$ is such that $\{T^n\}_{n=0}^{\infty}$ is b.u.e.m. (u.e.m.) at zero on $(L_p,\|\cdot\|_p)$. If $\mathcal{W}\subseteq W_1$, then 
$$bWW_p(\mathcal{W})=bWW_p(\overline{\mathcal{W}}) \ (\text{respectively, } WW_p(\mathcal{W})=WW_p(\overline{\mathcal{W}})).$$
\end{pro}
\begin{proof}
We will prove the case where $\{T^n\}_{n=0}^{\infty}$ is b.u.e.m. at zero on $(L_p,\|\cdot\|_p)$ only. Obviously, $bWW_p(\overline{\mathcal{W}})\subseteq bWW_p(\mathcal{W})$; hence, the other inclusion is all that remains to be proven.

Assume $x\in bWW_q(\mathcal{W})$ and $\epsilon>0$. Then, the fact that $x\in bWW_p(\mathcal{W})$ implies that there exists $e\in\mathcal{P}(\mathcal{M})$ such that $\tau(e^\perp)\leq\frac{\epsilon}{2}$ and $\{eM_n^\beta(T)(x)e\}_{n=1}^{\infty}$ converges in $\mathcal{M}$ for every $\beta\in\mathcal{W}$.

Since $\{T^n\}_{n=0}^{\infty}$ is b.u.e.m. at zero on $(L_p,\|\cdot\|_p)$, there exists $\gamma>0$ such that, whenever $\|y\|_p<\gamma$, there exists $f_y\in\mathcal{P}(\mathcal{M})$ such that $\tau(f_y^\perp)\leq\frac{\epsilon}{2}$ and $\sup_{n}\|f_yT^n(y)f_y\|_\infty\leq1.$ Let $\eta>0$ be such that $\eta\|x\|_p<\gamma$, and let $f\in\mathcal{P}(\mathcal{M})$ be the projection such that $\tau(f^\perp)\leq\frac{\epsilon}{2}$ and $\sup_{n}\|fT^n(x)f\|_\infty\leq\frac{1}{\eta}$.

Let $g=e\wedge f$. Then $\tau(g^\perp)\leq\epsilon$, $\sup_{n}\|gT^n(x)g\|_\infty\leq\frac{1}{\eta}$, and $\{gM_n^\beta(T)(x)g\}_{n=1}^{\infty}$ converges in $\mathcal{M}$ for every $\beta\in\mathcal{W}$.

Assume that $\alpha=\{\alpha_k\}_{k=0}^{\infty}\in\overline{\mathcal{W}}$ and $\delta>0$. Let $\beta=\{\beta_k\}_{k=0}^{\infty}\in\mathcal{W}$ be such that
$$\delta>\|\alpha-\beta\|_{W_1}=\limsup_{n\to\infty}\frac{1}{n}\sum_{k=0}^{n-1}|\alpha_k-\beta_k|.$$ Let $N_1$ be such that $\frac{1}{n}\sum_{k=0}^{n-1}|\alpha_k-\beta_k|<\delta$ for $n\geq N_1$.
Since $\beta\in\mathcal{W}$, we know that $\{gM_n^\beta(T)(x)g\}_{n=0}^{\infty}$ converges, so that the sequence is Cauchy. Let $N_2$ be such that $\|g(M_n^\beta(T)(x)-M_m^\beta(T)(x))g\|_\infty\leq\delta$ for $m,n\geq N_2$.
Let $N=\max\{N_1,N_2\}$. Then, for $m,n\geq N$, we find that
\begin{align*}
\|g(M_n^\alpha(T)(x)-M_m^\alpha(T)(x))g\|_\infty
&\leq\|gM_n^{\alpha-\beta}(T)(x)g\|_\infty
+\|gM_m^{\alpha-\beta}(T)(x)g\|_\infty \\
&\ \ \ \ \  +\|g(M_n^\beta(T)(x)-M_m^\beta(T)(x))g\|_\infty \\
&\leq \frac{1}{n}\sum_{k=0}^{n-1}|\alpha_k-\beta_k|\|gT^k(x)g\|_\infty \\
&\ \ \ \ \  +\frac{1}{m}\sum_{k=0}^{m-1}|\alpha_k-\beta_k|\|gT^k(x)g\|_\infty
+\delta \\
&\leq 2\delta\sup_{n\in\mathbb{N}_0}\|gT^n(x)g\|_\infty + \delta
\leq\left(\frac{2}{\eta}+1\right)\delta.
\end{align*}
Since $\delta>0$ was arbitrary, it follows that $\{gM_n^\alpha(T)(x)g\}_{n=1}^{\infty}$ is Cauchy, and hence, it is convergent in $\mathcal{M}$. Since $\alpha\in\overline{\mathcal{W}}$ is arbitrary, it follows that $x\in bWW_p(\overline{\mathcal{W}})$, proving the claim.
\end{proof}

A function $f:\mathbb{Z}\to\mathbb{C}$ is a {\it trigonometic polynomial} if there exists $\lambda_1,...,\lambda_n\in\mathbb{T}$ and $r_1,...,r_n\in\mathbb{C}$ such that $f(k)=\sum_{j=1}^{n}r_j\lambda_j^k$ for every $k\in\mathbb{Z}$. Note that $\{f(k)\}_{k=0}^\infty\in W_\infty$. A sequence $\alpha\in W_1$ will be said to be {\it$r$-Besicovich} if, for every $\delta>0$, there exists a trigonometric polynomial $f$ such that $\|\alpha-f\|_{W_r}<\delta$. In other words, $\alpha$ belongs to the $W_r$-seminorm closure of the trigonometric polynomials.

The convergence of averages weighted by \textit{bounded} Besicovich sequences were studied in the noncommutative setting in \cite{cls,cl1,li2}, to name a few. Having the assertion of Proposition \ref{p32} above, and in view of Proposition \ref{p22}, we deduce the following generalization of those results for $1$-Besicovich sequences.

\begin{cor}\label{c32}
Assume that $\mathcal{M}$ has a separable predual.
Let $1\leq p<\infty$ and $T\in DS^+$ be such that $\{T^n\}_{n=0}^{\infty}$ is b.u.e.m. (u.e.m.) at zero on $L_p$. Then, for every $1$-Besicovich sequence $\alpha$,
the averages $\{M_n^\alpha(T)(x)\}_{n=1}^{\infty}$ converge b.a.u. (respectively, a.u.) as $n\to\infty$ for every $x\in L_p(\mathcal{M},\tau)$.
\end{cor}

For general $T\in DS^+(\mathcal{M},\tau)$, one can still obtain a similar, but weaker, version of Proposition \ref{p33}. Since the core arguments are similar, we omit the proof.

\begin{cor}\label{c33}
Let $1\leq p<\infty$, $T\in DS^+(\mathcal{M},\tau)$, and $\mathcal{W}\subseteq W_\infty$. If $L_p(\mathcal{M},\tau)=bWW_p(\mathcal{W})$, then actually $L_p(\mathcal{M},\tau)=bWW_p(\overline{\mathcal{W}}\cap W_\infty)$. When $2\leq p<\infty$, the same holds true when $bWW_p$ is replaced with $WW_p$.
\end{cor}

As some immediate consequences of this result, one can extend the main result of \cite{li2} to hold for bounded Besicovich sequences (although this was already known by the results of \cite{hs}). One also finds that Corollary \ref{c31} above extends to bounded sequences $\alpha\in W_\infty$ that only converge to some $\ell(\alpha)\in\mathbb{C}$ along a subsequence of density $1$, which is a condition that occurs frequently when studying weakly mixing measure preserving dynamical systems. This also generalizes \cite[Theorem 3.1]{ob} by proving that one only needs a single projection for the a.u. or b.a.u. convergence of subsequential weighted averages of a bounded sequence along any density $1$ subsequence.

\section{Convergence of averages weighted by Hartman sequences}\label{s4}

In this section, we will consider operators $T\in DS^+(\mathcal{M},\tau)$ such that the iterates $\{T^n(x)\}_{n=0}^{\infty}$ converge b.a.u. or a.u. for every flight vector of $T$ on $L_q$ for some $1<q<\infty$ (or $2\leq q<\infty$ in the a.u.-case). It will be shown in Section \ref{s6} that this assumption is naturally satisfied in numerous settings. The Wiener-Wintner type theorems proved below concern some sets in $W_1$ that contain all bounded Hartman sequences, which as a special case means a noncommutative version of Bourgain's Return Times Theorem holds for these operators (see Corollary \ref{c41}).

A sequence $\alpha=\{\alpha_k\}_{k=0}^{\infty}\subset\mathbb{C}$ is called a \textit{Hartman sequence} if
$$\lim_{n\to\infty}\frac{1}{n}\sum_{k=0}^{n-1}\alpha_k\lambda^k=:c_\alpha(\overline{\lambda})$$ exists for every $\lambda\in\mathbb{T}$. It is known that $\{\lambda\in\mathbb{T}:c_\alpha(\lambda)\neq0\}$ is a countable set for every Hartman sequence $\alpha$. Let $H$ denote the set of all Hartman sequences in $W_1$. As noted in \cite{lot}, $H$ is a closed subspace of $W_1$. The $r$-Besicovich (bounded Besicovich) sequences discussed in Section \ref{s3} are an important class of Hartman sequences in $W_r$ for $1\leq r<\infty$ ($r=\infty$).

If $X$ is a Banach space, we denote its (continuous) dual space by $X^*$. A set $E\subseteq X$ is weakly conditionally compact if every sequence in $E$ contains a weakly convergent subsequence. In other words, for every sequence $\{x_n\}_{n=0}^{\infty}\subseteq E$, there exists a subsequence $\{x_{n_j}\}_{j=0}^{\infty}$ of $\{x_n\}_{n=0}^{\infty}$ and $x\in E$ such that $\{\phi(x_{n_j})\}_{j=0}^{\infty}$ converges to $\phi(x)$ in $\mathbb{C}$ for every $\phi\in X^*$. A linear operator $S:X\to X$ is \textit{weakly almost periodic} if $\{S^n(x):n\in\mathbb{N}_0\}$ is weakly conditionally compact for every $x\in X$.

Since $T$ is a contraction on $L_q$, it is also power-bounded on $L_q$. It is known that $L_q$ is a reflexive Banach space when $1<q<\infty$ (see \cite{ne}). From this, if $1<q<\infty$, it can be shown that $T$ is a weakly almost periodic operator on $L_q$.

The following is a summary of some of the results (namely Theorems 1.2 and 2.1) in \cite{lot} relating weakly almost periodic operators and Hartman sequences.

\begin{teo}\label{t41}\cite{lot} Let $\alpha=\{\alpha_k\}_{k=0}^{\infty}\in W_{1^+}\cap H$. Then for every weakly almost periodic operator $S$ on a Banach space $X$,
$$L(\alpha,S)(x):=\lim_{n\to\infty}\frac{1}{n}\sum_{k=0}^{n-1}\alpha_kS^k(x)\text{ exists in norm for every }x\in X.$$
Moreover, if $X$ is a Hilbert space and $S$ is a contraction, then, for every $x\in X$,
$$L(\alpha,S)(x)=\sum_{\lambda\in\mathbb{T}}c_{\alpha}(\lambda)E(\overline{\lambda})(x),$$ where, for every $\lambda\in\mathbb{T}$, $E(\lambda)$ is the projection of $X$ onto $\overline{(\lambda1-S)X}$ (noting only countably many terms of the sum are nonzero).
\end{teo}

This setup also allows the consideration of the Jacobs-de Leeuw-Glicksberg decomposition of $L_q$ for $1<q<\infty$ \cite[Chapter 16]{efhn}. This states that we may write $L_q(\mathcal{M},\tau)$ as
$$L_q(\mathcal{M},\tau)=\overline{\operatorname{span}(\mathcal{U}_q(T))}\oplus\mathcal{V}_q(T),$$ where
$$\mathcal{U}_q(T)=\Big\{x\in L_q(\mathcal{M},\tau):T(x)=\lambda x\text{ for some }\lambda\in\mathbb{T}\Big\},$$
$$\mathcal{V}_q(T)=\Big\{x\in L_q(\mathcal{M},\tau):\{T^{n_j}(x)\}_{j=0}^{\infty}\text{ converges weakly for some }\{n_j\}_{j=0}^{\infty}\subseteq\mathbb{N}_0\Big\},$$
with the closure of $\operatorname{span}(\mathcal{U}_q(T))$ being with respect to the norm on $L_q$. An element of $\mathcal{V}_q(T)$ is sometimes called a \textit{flight vector} of $T$. For notational convenience we define $\mathcal{U}_1(T)$ similarly for $L_1$.

A sequence $\alpha=\{\alpha_k\}_{k=0}^{\infty}\subset\mathbb{C}$ is called a \textit{linear sequence} if there exists a Banach space $X$ and a relatively weakly compact operator $S$ on $X$ with $z\in X$ and $\phi\in X^*$ such that $\alpha_k=\phi(S^k(z))$ for every $k\in\mathbb{N}_0$. It was shown in \cite{ei} that, for any linear sequence $\alpha$, there exists a Hartman sequence $\beta$ and $\gamma\in W_\infty$ with $\|\gamma\|_{W_1}=0$ such that $\alpha=\beta+\gamma$. All of these assumptions can be used to show that a linear sequence is a bounded Hartman sequences.

Observe that the setting we work in gives rise to numerous examples of linear sequences. To see this, let $T\in DS^+(\mathcal{M},\tau)$ and let $1<p,q<\infty$ be such that $\frac{1}{p}+\frac{1}{q}=1$. Then, for any $x\in L_p(\mathcal{M},\tau)$ and any $y\in L_q(\mathcal{M},\tau)$, the sequence $\{\tau(y^*T^n(x))\}_{n=0}^{\infty}$ is a linear sequence.

The following technical lemma justifies part of an assumption we will make later.

\begin{lm}\label{l41}
Assume $1<q<\infty$. If $x\in\mathcal{V}_q(T)$ is such that $T^n(x)\to y$ b.a.u. (a.u.) for some $y\in L_0$, then $y=0$.
\end{lm}
\begin{proof}
As a.u. convergence implies b.a.u. convergence, we will prove the b.a.u case.

Let $x\in\mathcal{V}_q(T)$ be such that $T^n(x)\to y$ b.a.u. for some $y\in L_0$. Since $\textbf{1}:=\{1\}_{k=0}^{\infty}$ is a bounded sequence and since $x$ is a flight vector for $T$, it follows by \cite[Theorem 4.1]{comlio} that $M_n(T)(x)\to0$ in the $L_q$-norm as $n\to\infty$, and so in measure in $L_0$ as well by \cite{ne}.

Fix $\epsilon>0$. Since $T^n(x)\to y$ b.a.u., there exists $e\in\mathcal{P}(\mathcal{M})$ such that $\tau(e^\perp)\leq\epsilon$ and $\|e(T^n(x)-y)e\|_\infty\to0$ as $n\to\infty$. With this, we find that
$$\|e(M_n(T)(x)-y)e\|_\infty=\left\|\frac{1}{n}\sum_{k=0}^{n-1}e(T^k(x)-y)e\right\|_\infty\leq\frac{1}{n}\sum_{k=0}^{n-1}\|e(T^k(x)-y)e\|_\infty.$$
Note that the right hand side of this inequality is the $n$-th term of the sequence of Ces\`{a}ro averages of $\{\|e(T^n(x)-y)e\|_\infty\}_{n=0}^{\infty}$, and so converges to $0$ as well. Consequently, $\{M_n(T)(x)\}_{n=1}^{\infty}$ converges to $y$ b.a.u. as $n\to\infty$ since $\epsilon>0$ is arbitrary. This convergence also occurs in measure, and since $L_0$ is a Hausdorff space with the measure topology, the uniqueness of limits in $L_0$ implies $y=0$.
\end{proof}

Since the computation regarding elements of $\mathcal{U}_p(T)$ will appear numerous times throughout the article, we will separate and write it as the following lemma.

\begin{lm}\label{l42}
For every $1\leq p<\infty$, $\operatorname{span}(\mathcal{U}_p(T))\subseteq WW_p(H)$.
\end{lm}
\begin{proof}
Assume $x\in\mathcal{U}_p(T)$ and $\epsilon>0$, and let $\lambda\in\mathbb{T}$ be such that $T(x)=\lambda x$. Then, since $x$ is a $\tau$-measurable operator, there exists $e\in\mathcal{P}(\mathcal{M})$ such that $\tau(e^\perp)\leq\epsilon$ and $xe\in\mathcal{M}$. Now, for every $\alpha=\{\alpha_k\}_{k=0}^{\infty}\in H$ and $n\geq 1$, we find
$$
\left\|(M_n^\alpha(T)(x)-c_\alpha(\overline{\lambda})x)e\right\|_\infty
=\left|\frac{1}{n}\sum_{k=0}^{n-1}\alpha_k\lambda^k-c_\alpha(\overline{\lambda})\right|\|xe\|_\infty,
$$ and since $\alpha$ is a Hartman sequence the last term tends to $0$ as $n\to\infty$. Therefore $\{M_n^\alpha(T)(x)e\}_{n=1}^{\infty}$ converges in $\mathcal{M}$, and since $\alpha\in H$ and $\epsilon>0$ was arbitrary, it follows that $x\in WW_p(H)$. Since $x\in\mathcal{U}_p(T)$ was arbitrary, it follows that $\mathcal{U}_p(T)\subseteq WW_p(H)$. Since $WW_p(\mathcal{A})$ is a subspace of $L_p$ for every $\mathcal{A}\subseteq W_1$, the result follows.
\end{proof}

\vfill

\begin{teo}\label{t42}
Let $\mathcal{M}$ be a von Neumann algebra with normal semifinite faithful trace $\tau$. Let $T\in DS^+(\mathcal{M},\tau)$ be such that, for some $1<q<\infty$ ($2\leq q<\infty$), $T^n(x)\to 0$ b.a.u. (respectively, a.u.) as $n\to\infty$ for every $x\in\mathcal{V}_q(T)$. Then, for every $1\leq p<\infty$ (respectively, $2\leq p<\infty$), $bWW_p(W_\infty\cap H)=L_p(\mathcal{M},\tau)$ (respectively, $WW_p(W_\infty\cap H)=L_p(\mathcal{M},\tau)$).
\end{teo}
\begin{proof}
We will prove the b.a.u. case only. Recall that, by the Jacobs-de Leeuw-Glicksberg decomposition of $L_q$, the set $\operatorname{span}(\mathcal{U}_q(T))+\mathcal{V}_q(T)$ is dense in $L_q$. As such, we will show that each summand is in $bWW_q(W_\infty\cap H)$.

We know that $\operatorname{span}(\mathcal{U}_q(T))\subseteq bWW_q(H)$ by Lemma \ref{l42}, and the definition of the spaces directly implies that $bWW_q(H)\subseteq bWW_q(W_\infty\cap H)$. Therefore $\operatorname{span}(\mathcal{U}_q(T))\subseteq bWW_q(W_\infty\cap H)$.

Assume $x\in\mathcal{V}_q(T)$. Assume $\epsilon>0$, and let $e\in\mathcal{P}(\mathcal{M})$ be such that $\tau(e^\perp)\leq\epsilon$ and $\|eT^n(x)e\|_\infty\to0$. Let $\beta=\{\beta_k\}_{k=0}^{\infty}\in W_\infty\cap H$. Then
$$
\left\|eM_n^\beta(T)(x)e\right\|_\infty
\leq\frac{1}{n}\sum_{k=0}^{n-1}|\beta_k|\|eT^k(x)e\|_\infty\leq\|\beta\|_{W_\infty}\frac{1}{n}\sum_{k=0}^{n-1}\|eT^k(x)e\|_\infty.$$ Since $\|eT^k(x)e\|_\infty\to0$ as $n\to\infty$, the last term tends to $0$ as $n\to\infty$ as well since it is the $n$-th term of the sequence of Ces\`{a}ro averages associated to the sequence. Therefore $\|eM_n^\beta(T)(x)e\|_\infty\to0$ as $n\to\infty$, and since $\epsilon>0$ and $\beta\in W_\infty\cap H$ were arbitrary, it follows that $x\in bWW_q(W_\infty\cap H)$, and so $\mathcal{V}_q(T)\subseteq bWW_q(W_\infty\cap H)$.

We have shown that $\operatorname{span}(\mathcal{U}_q(T))$ and $\mathcal{V}_q(T)$ are both subsets of $bWW_q(W_\infty\cap H)$, and since the latter set is a closed subspace of $L_q$ by Theorem \ref{t31}, it follows that $L_q(\mathcal{M},\tau)=bWW_q(W_\infty\cap H)$.

Now, fix $1\leq p<\infty$. Since $L_1\cap\mathcal{M}\subseteq bWW_q(W_\infty\cap H)$, we find that for every $x\in L_1\cap\mathcal{M}$ and $\epsilon>0$ there exists $e\in\mathcal{P}(\mathcal{M})$ such that $\tau(e^\perp)\leq\epsilon$ and $\{eM_n^\alpha(T)(x)e\}_{n=1}^{\infty}$ converges in $\mathcal{M}$ for every $\alpha\in W_\infty\cap H$. However, since $L_1\cap\mathcal{M}$ is also contained in $L_p$, this shows that $L_1\cap\mathcal{M}\subseteq bWW_p(W_\infty\cap H)$. Since $L_1\cap\mathcal{M}$ is dense in $L_p$, and since $bWW_p(W_\infty\cap H)$ is closed in $L_p$ by Proposition \ref{p31} and Theorem \ref{t31}, it follows that $bWW_p(W_\infty\cap H)=L_p$.
\end{proof}

In Proposition \ref{p32}, we made the assumption that $\{T^n\}_{n=0}^{\infty}$ is b.u.e.m. or u.e.m. at zero on $(L_q,\|\cdot\|_q)$ for some $1\leq q<\infty$ to obtain a maximal ergodic inequality for unbounded weights. We now make this additional assumption to obtain a stronger version of Theorem \ref{t42} on the fixed $L_q$-space. This property is usually satisfied for such $T$, since it is typically used to prove the b.a.u. or a.u. convergence of $\{T^n(x)\}_{n=0}^{\infty}$ for $x\in L_q$.

\begin{teo}\label{t43}
Let $\mathcal{M}$ be a von Neumann algebra with normal semifinite faithful trace $\tau$, and let $1<q<\infty$. Let $T\in DS^+(\mathcal{M},\tau)$ be such that $\{T^n\}_{n=0}^{\infty}$ is b.u.e.m. (u.e.m.) at zero on $(L_q,\|\cdot\|_q)$ and that $T^n(x)\to0$ b.a.u. (respectively, a.u.) as $n\to\infty$ for every $x\in\mathcal{V}_q(T)$. Then $bWW_q(W_{1^+}\cap H)=L_q(\mathcal{M},\tau)$ (respectively, $WW_q(W_{1^+}\cap H)=L_q(\mathcal{M},\tau)$).
\end{teo}
\begin{proof}
We will prove the theorem for $bWW_q(W_{1^+}\cap H)$ only. It suffices to prove that both $\operatorname{span}(\mathcal{U}_q(T))$ and $\mathcal{V}_q(T)$ are subsets of $bWW_q(W_{1^+}\cap H)$ due to Proposition \ref{p32} and Theorem \ref{t31} since their sum is dense in $L_q$ by construction.

By Lemma \ref{l42}, we know that $\operatorname{span}(\mathcal{U}_q(T))\subseteq bWW_q(H)$, which immediately implies that $\operatorname{span}(\mathcal{U}_q(T))\subseteq bWW_q(W_{1^+}\cap H)$ as well.

Assume that $x\in\mathcal{V}_q(T)$ and $\epsilon>0$. Let $e\in\mathcal{P}(\mathcal{M})$ be such that $\tau(e^\perp)\leq\epsilon$ and $\|eT^n(x)e\|_\infty\to0$ as $n\to\infty$.

Assume that $\alpha=\{\alpha_k\}_{k=0}^{\infty}\in\bigcup_{r>1}W_r$, and let $r,s>1$ be such that $\alpha\in W_r$ and $\frac{1}{r}+\frac{1}{s}=1$. Then H\"{o}lder's inequality implies
\begin{align*}
\|eM_n^\alpha(T)(x)e\|_\infty
&\leq\frac{1}{n}\sum_{k=0}^{n-1}|\alpha_k|\|eT^k(x)e\|_\infty \\
&\leq\frac{1}{n^{1/r}n^{1/s}}\left(\sum_{k=0}^{n-1}|\alpha_k|^r\right)^{1/r}\left(\sum_{k=0}^{n-1}\|eT^k(x)e\|_\infty^s\right)^{1/s} \\
&=\left(\frac{1}{n}\sum_{k=0}^{n-1}|\alpha_k|^r\right)^{1/r}\left(\frac{1}{n}\sum_{k=0}^{n-1}\|eT^k(x)e\|_\infty^s\right)^{1/s} \\
&\leq\left(\sup_{m\in\mathbb{N}}\frac{1}{m}\sum_{k=0}^{m-1}|\alpha_k|^r\right)^{1/r}\left(\frac{1}{n}\sum_{k=0}^{n-1}\|eT^k(x)e\|_\infty^s\right)^{1/s}.
\end{align*}
Since $\|eT^n(x)e\|_\infty\to0$ as $n\to\infty$, $\|eT^n(x)e\|_\infty^s\to0$ as well, and so the Ces\'{a}ro averages of this will converge to the same value too. Therefore the final expression tends to $0$ as $n\to\infty$, which implies that $\|eM_n^\alpha(T)(x)e\|_\infty\to0$ as $n\to\infty$. Since $\epsilon>0$ and $\alpha\in\bigcup_{r>1}W_r$ are arbitrary, it follows by Proposition \ref{p33} that $x\in bWW_q(\bigcup_{r>1}W_r)=bWW_q(W_{1^+})\subseteq bWW_q(W_{1^+}\cap H)$, from which we obtain $\mathcal{V}_q(T)\subseteq bWW_q(W_{1^+}\cap H)$.

We have proven that $\operatorname{span}(\mathcal{U}_q(T)),\mathcal{V}_q(T)\subseteq bWW_q(W_{1^+}\cap H)$. Therefore, since the sum of these two spaces is dense in $L_q$, and since $bWW_q(W_{1^+}\cap H)$ is closed in $L_q$ by Proposition \ref{p32} and Theorem \ref{t31}, it follows that $L_q\subseteq bWW_q(W_{1^+}\cap H)$, and since the latter space is a subset of the former by definition, we find that $L_q(\mathcal{M},\tau)=bWW_q(W_{1^+}\cap H)$.
\end{proof}

An even stronger conclusion of this result can be obtained for subsets of $W_r$ which are bounded in its seminorm when $x\in\mathcal{V}_q(T)$. In particular, since $$\|eM_n^\alpha(T)(x)e\|_\infty\leq|\alpha|_r\left(\frac{1}{n}\sum_{k=0}^{n-1}\|eT^k(x)e\|_\infty^s\right)^{1/s},$$ for $x\in\mathcal{V}_q(T)$ and $\alpha\in W_r$, one obtains the following analogue of Bourgain's uniform Wiener-Wintner ergodic theorem.

\begin{cor}\label{c42}
Assume $1\leq p<\infty$, $1<r\leq\infty$, and $0<b<\infty$. Suppose $x\in L_p(\mathcal{M},\tau)$ is such that $T^n(x)\to0$ b.a.u. (a.u.) as $n\to\infty$. Then, for every $\epsilon>0$, there exists $e\in\mathcal{P}(\mathcal{M})$ such that $\tau(e^\perp)\leq\epsilon$ and
$$\lim_{n\to\infty}\sup_{|\alpha|_{W_r}\leq b}\|eM_n^\alpha(T)(x)e\|_\infty=0 \  (\text{respectively, }\lim_{n\to\infty}\sup_{|\alpha|_{W_r}\leq b}\|M_n^\alpha(T)(x)e\|_\infty=0).$$
\end{cor}

\begin{ex}\label{e41}
We recall again a common construction of sequences in $W_{1^+}\cap H$. Let $(X,\mathcal{F},\mu)$ be a probability space, and let $\phi:X\to X$ be an ergodic measure-preserving transformation. Let $1< r\leq\infty$ and $f\in L_r(X,\mathcal{F},\mu)$. Then, by \cite[Proposition 1.5]{lot}, the sequence $\{f(\phi^k(\omega))\}_{k=0}^{\infty}$ is in $W_r\cap H$ for $\mu$-a.e. $\omega\in X$. Furthermore, if $\phi$ is weakly mixing and $f$ is not a $\mu$-a.e. constant function, then for $\mu$-a.e. $\omega\in X$ the sequence is Hartman but not $1$-Besicovitch.
\end{ex}

By focusing our attention to the cases where $f\in L_\infty(X,\mathcal{F},\mu)$, we obtain the class of weights involved in Bourgain's Return Times Theorem. With this in mind, we state a noncommutative version of that result in the setting of Theorem \ref{t42}.

\begin{cor}\label{c41}
Assume $1\leq p<\infty$ ($2\leq p<\infty$), and let $T\in DS^+(\mathcal{M},\tau)$ be such that, for some $1<q<\infty$ (respectively, $2\leq q<\infty$), $T^n(x)\to0$ b.a.u. (respectively, a.u.) as $n\to\infty$ for every $x\in\mathcal{V}_q(T)$. Then, for every $x\in L_p(\mathcal{M},\tau)$ and $\epsilon>0$, there exists $e\in\mathcal{P}(\mathcal{M})$ such that $\tau(e^\perp)\leq\epsilon$ and, given a probability measure-preserving system $(X,\mathcal{F},\mu,\phi)$ and $f\in L_\infty(X,\mathcal{F},\mu)$, one has that
$$e\left[\frac{1}{n}\sum_{k=0}^{n-1}f(\phi^k(s))T^k(x)\right]e \ \left(\text{respectively, } \left[\frac{1}{n}\sum_{k=0}^{n-1}f(\phi^k(s))T^k(x)\right]e\right)$$ converges in $\mathcal{M}$ for $\mu$-a.e. $s\in X$. In particular, for every $E\in\mathcal{F}$ with $\mu(E)>0$ and $s\in E$, letting $n_0=\inf\{j\in\mathbb{N}_0:\phi^j(s)\in E\}$ and $n_k=\inf\{j\in\mathbb{N}_0:j>n_{k-1} \text{ and }\phi^j(x)\in E\}$, the subsequential averages
$$e\left[\frac{1}{n}\sum_{k=0}^{n-1}T^{n_k}(x)\right]e \ \left(respectively, \left[\frac{1}{n}\sum_{k=0}^{n-1}T^{n_k}(x)\right]e\right)$$ converge in $\mathcal{M}$ as $n\to\infty$ for every $x\in L_p(\mathcal{M},\tau)$.
\end{cor}

In searching the literature, the author was unable to find any reference to the results above regarding the pointwise/almost uniform convergence of averages weighted by every sequence in $W_{1^+}\cap H$ in the commutative setting. As such, we will now translate our results to that setting. If $(X,\mathcal{F},\mu)$ is a measure space, recall that Egorov's theorem states that $\mu$-a.e. convergence of a sequence of functions is equivalent to a.u. convergence if $\mu(X)<\infty$. However, a.u. convergence is stronger than $\mu$-a.e. convergence when $\mu(X)=\infty$.

\begin{cor}\label{c43}
Let $(X,\mathcal{F},\mu)$ be a $\sigma$-finite measure space, and let $\mathcal{M}=L_\infty(X,\mathcal{F},\mu)$ be the commutative von Neumann algebra acting on the Hilbert space $L_2(X,\mathcal{F},\mu)$ by multiplication operators. Let $\tau(f)=\int_{X}f\,d\mu$ for $f\in L_\infty(X,\mathcal{F},\mu)^+$ be the usual n.s.f. trace on $\mathcal{M}$.

Fix $1<q<\infty$, and assume that $T\in DS^+(\mathcal{M},\tau)$ be such that $T^n(f)\to0$ a.u. for every $f\in\mathcal{V}_q(T)$. Then, for every $1\leq p<\infty$, $f\in L_p(X,\mathcal{F},\mu)$, and $\epsilon>0$, there exists $Y\in\mathcal{F}$ with $\mu(Y^c)<\epsilon$ such that $\{\frac{1}{n}\sum_{k=0}^{n-1}\alpha_kT^k(f)\chi_{Y}\}_{n=1}^{\infty}$ converges uniformly for every $\alpha\in W_\infty\cap H$.

Furthermore, if the iterates $\{T^n\}_{n=0}^{\infty}$ are u.e.m. at zero on $(L_q(X,\mathcal{F},\mu),\|\cdot\|_q)$, then $\{\frac{1}{n}\sum_{k=0}^{n-1}\alpha_k T^k(f)\chi_{Y}\}_{n=1}^{\infty}$ converges uniformly for every $f\in L_q(X,\mathcal{F},\mu)$ and for every $\alpha\in W_{1^+}\cap H$.
\end{cor}

In the setting of Theorem \ref{t43}, for any $x\in L_q$ and $\alpha\in W_{1^+}\cap H$, we now know that $\{M_n^\alpha(T)(x)\}_{n=1}^{\infty}$ converges b.a.u. (a.u.) as $n\to\infty$. Since $L_0$ is complete with respect to b.a.u. (a.u.) convergence by \cite[Theorem 2.3]{cls}, we know that the b.a.u. (a.u.) limit $\widehat{x}_\alpha$ of this sequence is also in $L_0$. The next proposition says even more, namely, that $\widehat{x}_\alpha$ is actually in $L_q$ as well.

\begin{pro}\label{p41}
Assume $1\leq p<\infty$, $T\in DS^+(\mathcal{M},\tau)$, and $\alpha=\{\alpha_k\}_{k=0}^{\infty}\in W_1$. Suppose, for some $x\in L_p(\mathcal{M},\tau)$, that $\{M_n^\alpha(T)(x)\}_{n=1}^{\infty}$ converges b.a.u. or a.u. to some $\widehat{x}\in L_0(\mathcal{M},\tau)$ as $n\to\infty$. Then $\widehat{x}\in L_p(\mathcal{M},\tau)$.
\end{pro}
\begin{proof}
Since a.u. convergence implies b.a.u. convergence to the same element in $L_0$, it suffices to prove the b.a.u. case.

Since $\alpha\in W_1$, we know that
$|\alpha|_{W_1}<\infty$. For every $n\in\mathbb{N}$, since $T\in DS^+$ implies that $T:L_p\to L_p$ is a contraction, we have that
$$\|M_n^\alpha(T)(x)\|_p\leq\frac{1}{n}\sum_{k=0}^{n-1}|\alpha_k|\|T^k(x)\|_p\leq\|x\|_p\frac{1}{n}\sum_{k=0}^{n-1}|\alpha_k|\leq |\alpha|_{W_1}\|x\|_p.$$ Therefore $M_n^\alpha(T)(x)$ is in the closed ball of radius $|\alpha|_{W_1}\|x\|_p$ in $L_p$ for each $n\geq1$. Since the closed unit ball of $L_p$ is closed with respect to the measure topology (see \cite[Theorem 1.2]{cls}), and since $M_n^\alpha(T)(x)\to\widehat{x}$ b.a.u. implies the convergence takes place in measure, it follows that $\widehat{x}\in L_p(\mathcal{M},\tau)$.
\end{proof}

\begin{rem}\label{r43}
One can extend Theorem \ref{t42} and Corollary \ref{c31} to noncommutative symmetric and fully symmetric spaces. (See \cite{cl2} for a more thorough discussion of these spaces.)

Let $E\subseteq L_0$ be a (non-zero) Banach space with norm $\|\cdot\|_E$. Then $E$ is called a \textit{symmetric (fully symmetric)} space on $(\mathcal{M},\tau)$ if
$$x\in E, \ y\in L_0, \ \mu_t(y)\leq \mu_t(x) \text{ for all }t>0$$
$$(\text{respectively, }x\in E, y\in L_0, \int_{0}^{s}\mu_t(y)dt\leq\int_{0}^{s}\mu_t(x)dt \text{ for all }s>0)$$ implies that $y\in E$ and $\|y\|_E\leq\|x\|_E$. The first examples of symmetric spaces on $\mathcal{M}$ that one typically encounters are the noncommutative $L_p$-spaces associated with $\mathcal{M}$ for every $1\leq p<\infty$, although there are many others of interest as well.

For each $x\in L_1+\mathcal{M}$ write
$$\|x\|_{L_1+\mathcal{M}}:=\inf\{\|y\|_1+\|z\|_\infty:x=y+z\text{ for some }y\in L_1 \text{ and }z\in\mathcal{M}\}.$$
Then $\|\cdot\|_{L_1+\mathcal{M}}$ is a norm on $L_1+\mathcal{M}$ under which it becomes a Banach space.

Define the space $\mathcal{R}_{\tau}\subseteq L_1+\mathcal{M}$ as
$$\mathcal{R}_{\tau}:=\{x\in L_1+\mathcal{M}:\mu_t(x)\to0\text{ as }t\to\infty\}.$$ It is known that $\mathcal{R}_{\tau}$ is the closure of $L_1\cap\mathcal{M}$ with respect to the norm $\|\cdot\|_{L_1+\mathcal{M}}$, so that $(\mathcal{R}_{\tau},\|\cdot\|_{L_1+\mathcal{M}})$ is a Banach space. Moreover, it is actually a fully symmetric space. Some other important properties of $\mathcal{R}_{\tau}$ are the following: if $\tau(1)<\infty$, then $\mathcal{R}_{\tau}=L_1$, and if $\tau(1)=\infty$, then \cite[Proposition 2.2]{cl2} states that a fully symmetric space $E$ is contained in $\mathcal{R}_{\tau}$ if and only if $1\not\in E$.

Following the proofs of Theorem 4.3 and 4.6 of \cite{cl2} almost identically, one can show that, if $\mathcal{W}\subseteq W_\infty$, $T\in DS^+(\mathcal{M},\tau)$, and $L_1(\mathcal{M},\tau)=bWW_1(\mathcal{W})$, then for every $x\in\mathcal{R}_{\tau}$ and $\epsilon>0$ there exists $e\in\mathcal{P}(\mathcal{M})$ such that $\tau(e^{\perp})\leq\epsilon$ and $\{eM_n^\alpha(T)(x)e\}_{n=1}^{\infty}$ converges in $\mathcal{M}$ for every $\alpha\in\mathcal{W}$. In particular, this applies for $\mathcal{W}=\mathcal{W}_c$ by Corollary \ref{c31} and $\mathcal{W}=W_\infty\cap H$ for the operators $T\in DS^+$ which were considered in Theorem \ref{t42}.
\end{rem}

\section{Convergence along subsequences of density zero}\label{s5}

In this section, under the same assumptions we made in the previous section we will consider some subsequential ergodic theorems along sequences with density zero. In particular, we will follow \cite{ei} and consider the b.a.u. and a.u. convergence of the averages of powers of the operators $T$ considered in Theorem \ref{t43} along the primes, and as a consequence of this show that the von Mangoldt function works well as a weight for the noncommutative individual ergodic theorem for such $T$. We also prove a Wiener-Wintner type ergodic theorem for moving average sequences.

Let $\Lambda:\mathbb{N}_0\to\mathbb{C}$ be the von Mangoldt function, which is given by
$$\Lambda(n)=\left\{\begin{array}{cl}\ln(p),&\text{ when }n=p^k\text{ for some }k\in\mathbb{N}\text{ and }p\text{ is prime} \\ 0,&\text{ otherwise}\end{array}\right.$$
The norm convergence of the averages of $\{S^n(\xi)\}_{n=0}^{\infty}$ when weighted by $\{\Lambda(n)\}_{n=0}^{\infty}$ under numerous conditions on a bounded operator $S$ on a Banach space $X$ with $\xi\in X$ was studied in Section 4 of \cite{el}. We will often write $\Lambda=\{\Lambda(n)\}_{n=0}^{\infty}$.

The prime number theorem is equivalent to the statement $\frac{1}{n}\sum_{k=0}^{n-1}\Lambda(k)\to 1$ as $n\to\infty$, and since the former is known to be true, this implies that $\Lambda\in W_1$ since $\Lambda(k)\geq0$ for every $k$.
Also, it is known that $\Lambda$ is a Hartman sequence by \cite[Theorem 4.3]{el}. However, $\|\Lambda-\alpha\|_{W_1}\geq\frac{1}{2}$ for every $\alpha\in W_r$ and $1<r\leq\infty$ by \cite[Proposition 4.6]{el}, so that $\Lambda\not\in W_{1^+}\cap H$. Therefore, we study the averages of $T$ weighted by $\Lambda$ here.

The following technical result, which is essentially a maximal inequality for subsequences of $\mathbb{N}_0$ (note the similarities with Proposition \ref{p32}), will be needed later.

\begin{pro}\label{p51}
Let $1\leq p<\infty$ and $T\in DS^+$ be such that $\{T^n\}_{n=0}^{\infty}$ is b.u.e.m. (u.e.m.) at zero on $(L_p,\|\cdot\|_p)$. Then, for every $\epsilon,\delta>0$, there exists $\gamma>0$ such that, for any $x\in L_p$ with $\|x\|_p<\gamma$, there exists $e\in\mathcal{P}(\mathcal{M})$ such that $\tau(e^\perp)\leq\epsilon$ and, for any strictly increasing sequence $\{k_j\}_{j=0}^{\infty}\subseteq\mathbb{N}_0$,
$$\sup_{n\in\mathbb{N}}\left\|e\left(\frac{1}{n}\sum_{j=0}^{n-1}T^{k_j}(x)\right)e\right\|_\infty\leq\delta \ (\text{respectively, }\sup_{n\in\mathbb{N}}\left\|\left(\frac{1}{n}\sum_{j=0}^{n-1}T^{k_j}(x)\right)e\right\|_\infty\leq\delta).$$
\end{pro}
\begin{proof}
Assume $\epsilon,\delta>0$, and let $\gamma>0$ be as in the definition for $\{T^n\}_{n=0}^{\infty}$ being b.u.e.m. at zero on $(L_p,\|\cdot\|_p)$. Let $x\in L_p$ be such that $\|x\|_p<\gamma$, and let $e\in\mathcal{P}(\mathcal{M})$ be the corresponding projection so that $\tau(e^\perp)\leq\epsilon$ and $\sup_{n}\|eT^n(x)e\|_\infty\leq\delta$.
	
Let $\{k_j\}_{j=0}^{\infty}\subseteq\mathbb{N}_0$ be a strictly increasing sequence. Then, for any $n\in\mathbb{N}$,
$$\left\|\frac{1}{n}\sum_{j=0}^{n-1}eT^{k_j}(x)e\right\|_\infty\leq\frac{1}{n}\sum_{j=0}^{n-1}\|eT^{k_j}(x)e\|_\infty\leq\frac{1}{n-1}\sum_{j=0}^{n-1}\delta=\delta.$$
As $n\in\mathbb{N}$ was arbitrary, the result holds.
\end{proof}

Suppose a sequence $\{a_n\}_{n=0}^{\infty}\subset\mathbb{C}$ converges to $a$. Then any sequence obtained by repeating terms finitely many times in a row will also converge to $a$ as well, although at a possibly slower rate. In particular, if $\pi(n)$ denotes the number of primes less than or equal to $n$, then $a_n\to a$ as $n\to\infty$ implies that $a_{\pi(n)}\to a$.

\begin{teo}\label{t51}
Let $1<q<\infty$ and $T\in DS^+$ be such that $\{T^n\}_{n=0}^{\infty}$ is b.u.e.m. (u.e.m.) at zero on $(L_q,\|\cdot\|_q)$ and $T^n(x)\to0$ b.a.u. (a.u.) for every $x\in\mathcal{V}_q(T)$. Then, if $\{p_k\}_{k=0}^{\infty}$ denotes the increasing enumeration of the primes in $\mathbb{N}_0$, and $\Lambda$ denotes the sequence determined by the von Mangoldt function, then the sequences
$$\frac{1}{n}\sum_{k=0}^{n-1}T^{p_k}(x), \ \text{ and } \  M_n^\Lambda(T)(x)=\frac{1}{n}\sum_{k=0}^{n-1}\Lambda(k)T^k(x)$$ both converge b.a.u. (a.u.) for every $x\in L_q$.
\end{teo}
\begin{proof}
If $x\in\mathcal{U}_q(T)$ and $\epsilon>0$, then $T(x)=\lambda x$ for some $\lambda\in\mathbb{T}$. As such, $$\frac{1}{n}\sum_{k=0}^{n-1}T^{p_k}(x)=\left(\frac{1}{n}\sum_{k=0}^{n-1}\lambda^{p_k}\right)x.$$
It is known that $\{\frac{1}{n}\sum_{k=0}^{n-1}\lambda^{p_k}\}_{n=1}^{\infty}$ converges for every $\lambda\in\mathbb{T}$ by \cite[Theorem 4.3]{el}. Hence, the scalar sequence converges, and so choosing $e\in\mathcal{P}(\mathcal{M})$ with $\tau(e^\perp)\leq\epsilon$ and $xe\in\mathcal{M}$ completes this case similar to before.

Now assume $x\in\mathcal{V}_q(T)$ and $\epsilon>0$, and let $e\in\mathcal{P}(\mathcal{M})$ be such that $\tau(e^\perp)\leq\epsilon$ and $\|eT^n(x)e\|_\infty\to0$. Then the subsequence $\|eT^{p_n}(x)e\|_\infty\to0$; consequently, the Ces\`{a}ro averages $\frac{1}{n}\sum_{k=0}^{n-1}\|eT^{p_k}(x)e\|_\infty\to0$ as well. Consequently,
$$\left\|e\left(\frac{1}{n}\sum_{k=0}^{n-1}T^{p_k}(x)\right)e\right\|_\infty\leq\frac{1}{n}\sum_{k=0}^{n-1}\|eT^{p_k}(x)e\|_\infty\to0.$$
Since $x\in\mathcal{V}_q(T)$ was arbitrary, it follows by Proposition \ref{p51} and Proposition \ref{p22} that these subsequential averages converge b.a.u. for every $x\in L_q$.

Now we prove the b.a.u. convergence of $\{M_n^\Lambda(T)(x)\}_{n=1}^{\infty}$ for every $x\in L_q$. Since $\Lambda\in H\subset W_1$, it follows by Lemma \ref{l42} that $\text{span}(\mathcal{U}_q(T))\subseteq bWW_q(\{\Lambda\})$; hence, we just need to prove the claim for $x\in\mathcal{V}_q(T)$.

Assume $x\in\mathcal{V}_q(T)$ and observe that that
$$\left\|e\left(\frac{1}{n}\sum_{k=0}^{n-1}\Lambda(k)T^k(x)\right)e\right\|_\infty\leq\frac{1}{n}\sum_{k=0}^{n-1}\Lambda(k)\|eT^k(x)e\|_\infty.$$
Using $b_n=\|eT^n(x)e\|_\infty$ in \cite[Lemma 4.1]{el}, noting that $\sup_{n}\|eT^n(x)e\|_\infty<\infty$ since the sequence is convergent to $0$, we find that $$\left|\frac{1}{n}\sum_{k=0}^{n-1}\Lambda(k)\|eT^k(x)e\|_\infty-\frac{1}{\pi(n)}\sum_{k=0}^{\pi(n)-1}\|eT^{p_k}(x)e\|_\infty\right|\to0$$ as $n\to\infty$. Since $\frac{1}{\pi(n)}\sum_{k=0}^{\pi(n)-1}\|eT^{p_k}(x)e\|_\infty\to0$, the comment before the theorem implies that $\frac{1}{n}\sum_{k=0}^{n-1}\Lambda(k)\|eT^k(x)e\|_\infty\to0$. Therefore $eM_n^\Lambda(T)(x)e\to0$ in $\mathcal{M}$, and since $\epsilon>0$ and $x\in\mathcal{V}_q(T)$ are arbitrary, it follows that $\mathcal{V}_q(T)\subseteq bWW_q(\{\Lambda\})$.

By Proposition \ref{p32} we know that the weighted averages $\{M_n^\Lambda(T)\}_{n=1}^{\infty}$ are b.u.e.m. at zero on $(L_q,\|\cdot\|_q)$. Therefore,  since $\text{span}(\mathcal{U}_q(T))+\mathcal{V}_q(T)$ is dense in $L_q$, we apply Theorem \ref{t31} to obtain that $bWW_q(\{\Lambda\})=L_q$, which proves the claim.
\end{proof}

These assumptions also work in regards to studying the b.a.u. and a.u. convergence of averages along moving average sequences.

A moving average sequence is a sequence $w=\{(k_n,m_n)\}_{n=1}^{\infty}$, where $m_n\in\mathbb{N}_0$, $k_n\in\mathbb{N}$, and $k_n\to\infty$ as $n\to\infty$. Given $T\in DS^+$ and $x\in L_p$ with $1\leq p<\infty$, the corresponding averages are denoted by
$$M_{w,n}(T)(x)=\frac{1}{k_n}\sum_{j=0}^{k_n-1}T^{m_n+j}(x).$$ The usual ergodic averages correspond to $k_n=n$ and $m_n=0$ for every $n\in\mathbb{N}$.

The proof of the following technical result regarding the averages along these types of sequences is very similar to that of Proposition \ref{p51}; hence, it is omitted.

\begin{pro}\label{p52}
Let $1<q<\infty$ and $T\in DS^+$ be such that $\{T^n\}_{n=0}^{\infty}$ is b.u.e.m. (u.e.m.) at zero on $(L_q,\|\cdot\|_q)$. Then, for every $\epsilon,\delta>0$, there exists $\gamma>0$ such that, for any $x\in L_q$ with $\|x\|_q<\gamma$, there exists $e\in\mathcal{P}(\mathcal{M})$ such that $\tau(e^\perp)\leq\epsilon$ and, for every moving average sequence $w$,
$$\sup_{n\in\mathbb{N}}\|eM_{w,n}(T)(x)e\|_\infty\leq\delta \  (\text{respectively, }\sup_{n\in\mathbb{N}}\|M_{w,n}(T)(x)e\|_\infty\leq\delta).$$
\end{pro}

\begin{teo}\label{t52}
Let $1<q<\infty$ and $T\in DS^+$ be such that $\{T^n\}_{n=0}^{\infty}$ is b.u.e.m. (u.e.m.) at zero on $(L_q,\|\cdot\|_q)$ and $T^n(x)\to0$ b.a.u. (respectively, a.u.) for every $x\in\mathcal{V}_q(T)$. Then, for every $x\in L_q(\mathcal{M},\tau)$ and every $\epsilon>0$, there exists $e\in\mathcal{P}(\mathcal{M})$ such that $\tau(e^\perp)\leq\epsilon$ and the sequence $\{eM_{w,n}(T)(x)e\}_{n=1}^{\infty}$ (respectively, $\{M_{w,n}(T)(x)e\}_{n=1}^{\infty}$) converges in $\mathcal{M}$ for every moving average sequence $w$.
\end{teo}
\begin{proof}
We prove the b.a.u. version only. Let $\mathcal{C}$ denote the set of all $x\in L_q$ such that, for every $\epsilon>0$, there exists $e\in\mathcal{P}(\mathcal{M})$ such that $\tau(e^\perp)\leq\epsilon$ and $\{eM_{w,n}(T)(x)e\}_{n=1}^{\infty}$ converges in $\mathcal{M}$ for every moving average sequence $w$.

Assume $x\in\mathcal{U}_q(T)$ and $\epsilon>0$, and let $\lambda\in\mathbb{T}$ be such that $T(x)=\lambda x$. Let $w=\{(k_n,m_n)\}_{n=1}^{\infty}$ be a moving average sequence. Note that $T^n(x)\in\text{span}(\{x\})$ for every $n\in\mathbb{N}_0$, meaning $M_{w,n}(T)(x)\in\text{span}(\{x\})$ for every $n\geq0$. By \cite[Theorem 2.1]{comli}, we know that $M_{w,n}(T)(x)$ converges in the $L_q$-norm to some $\widehat{x}\in L_q$. However, since the entire sequence is in $\text{span}(\{x\})$, and this space is a closed subspace of $L_q$ since it is finite dimensional, it follows that $\widehat{x}\in\text{span}(\{x\})$, so that  $\widehat{x}=\mu_w x$ for some $\mu_w\in\mathbb{C}$. Notice that the problem has been reduced to studying the scalar case of $\frac{1}{k_n}\sum_{j=0}^{k_n-1}\lambda^{m_n+j}\to\mu_w$ as $n\to\infty$.

Let $e\in\mathcal{P}(\mathcal{M})$ be such that $\tau(e^\perp)\leq\epsilon$ and $xe\in\mathcal{M}$. Then
$$\lim_{n\to\infty}\|(M_{w,n}(T)(x)-\widehat{x})e\|_\infty
=\lim_{n\to\infty}\left|\frac{1}{k_n}\sum_{j=0}^{k_n-1}\lambda^{m_n+j}-\mu_w\right|\|xe\|_\infty
=0.$$
Since the moving average sequence $w$ and $\epsilon>0$ were arbitrary, it follows that $x\in\mathcal{C}$, and since $x$ was arbitrary, it follows that $\operatorname{span}(\mathcal{U}_q(T))\subseteq\mathcal{C}$.

Now assume that $x\in\mathcal{V}_q(T)$ and $\epsilon>0$, and let $w=\{(k_n,w_n)\}_{n=1}^{\infty}$ be a moving average sequence. Let $e\in\mathcal{P}(\mathcal{M})$ be such that $\|eT^n(x)e\|_\infty\to0$ as $n\to\infty$. Let $S=\sup_{n\in\mathbb{N}_0}\|eT^n(x)e\|_\infty$, noting that this value is finite since the sequence is convergent.

Assume $\delta>0$. Let $N_1\in\mathbb{N}$ be such that $n\geq N_1$ implies $\|eT^n(x)e\|_\infty\leq\frac{\delta}{2}$ (since the sequence goes to $0$). Let $N\geq N_1$ be such that $\frac{1}{k_n}<\frac{\delta}{2N_1 (S+1)}$ for $n\geq N$ (since $k_n\to\infty$ as $n\to\infty$). Then we find that, for $n\geq N$,
\begin{align*}
\|eM_{w,n}(T)(x)e\|_\infty
&\leq\frac{1}{k_n}\sum_{j=0}^{k_n-1}\|eT^{m_n+j}(x)e\|_\infty \\
&=\frac{1}{k_n}\sum_{\substack{0\leq j\leq k_n-1\\ m_n+j<N_1}}\|eT^{m_n+j}(x)e\|_\infty + \frac{1}{k_n}\sum_{\substack{0\leq j\leq k_n-1\\ m_n+j\geq N_1}}\|eT^{m_n+j}(x)e\|_\infty \\
&\leq\frac{1}{k_n}\sum_{\substack{0\leq j\leq k_n-1\\ m_n+j<N_1}}S +\frac{1}{k_n}\sum_{\substack{0\leq j\leq k_n-1\\ m_n+j\geq N_1}}\frac{\delta}{2}
\leq\frac{N_1S}{k_n}+\frac{k_n\delta}{2k_n}\leq\frac{\delta}{2}+\frac{\delta}{2}=\delta.
\end{align*}
Since $\delta>0$ was arbitrary, it follows that $\|eM_{w,n}(T)(x)e\|_\infty\to0$ as $n\to\infty$, and so $x\in\mathcal{C}$. Since $x$ was arbitrary, it follows that $\mathcal{V}_q(T)\subseteq\mathcal{C}$.

Therefore we have shown that $\text{span}(\mathcal{U}_q(T))+\mathcal{V}_q(T)$ is contained in $\mathcal{C}$. The fact that $\mathcal{C}$ is closed in $L_q$ follows from an argument similar to the proof of Theorem \ref{t31} via Proposition \ref{p52}. Hence, it follows that $L_q(\mathcal{M},\tau)=\mathcal{C}$.
\end{proof}

\begin{rem}\label{r51}
This statement can be rewritten as in the case of Corollary \ref{c42}, which seems to be a new result of its kind in the commutative setting.
\end{rem}

\section{Examples}\label{s6}

There are functions on a probability space whose averages with respect to the Koopman operator of a measure-preserving transformation fail to converge almost everywhere when weighted by certain bounded Hartman sequences. As our results concern this class of weights, this warrants the necessity of making some assumptions on $T\in DS^+(\mathcal{M},\tau)$.

In this section, we will discuss known conditions on the operator $T\in DS^+(\mathcal{M},\tau)$ for which the assumptions we made hold. We will also provide concrete examples of operators that are known to actually satisfy those conditions. In Propositions \ref{p61} and \ref{p62} we show that the assumptions made on $T$ throughout the article are equivalent to $T^k$ satisfying the same assumptions for some $k\in\mathbb{N}$, which expands the list of such operators (see Example \ref{e65}).

Throughout this section, once $T\in DS^+$ is fixed, we will write $\mathcal{T}=\{T^n\}_{n=0}^{\infty}$. In order to state the results in a unified way from the articles they are adopted from, we will introduce some notation from the theory of noncommutative vector valued $L_p$-spaces. However, we will only use these to define weak type and strong type maximal inequalities for $\mathcal{T}$, acknowledging that the definitions can be given for more general sequences of operators on $L_1+\mathcal{M}$.
 
Fix $1\leq p\leq\infty$. Let $L_p(\mathcal{M};\ell_\infty)$ be the space of all sequences $\{x_n\}_{n=0}^{\infty}\subset L_p$ such that there exists $a,b\in L_{2p}$ and  $\{y_n\}_{n=0}^{\infty}\subset\mathcal{M}$ with $\sup_{n}\|y_n\|_\infty<\infty$ such that $x_n=ay_nb$ for every $n\geq0$. Then $L_p(\mathcal{M};\ell_\infty)$ can be shown to be a Banach space with respect to the norm
$$\|\{x_n\}_{n=0}^{\infty}\|_{L_p(\mathcal{M};\ell_\infty)}:=\inf_{x_n=ay_nb}\|a\|_{2p}\|b\|_{2p}\sup_{n\in\mathbb{N}_0}\|y_n\|_\infty,$$ where the infimum is over all factorizations of $x$ as defined above. For more information on these spaces, one can refer to \cite{dj,jx}. 
Similarly, $L_p(\mathcal{M};\ell_{\infty}^{c})$ will denote the space of all sequences $\{x_n\}_{n=1}^{\infty}\subset L_p$ such that there exists $a\in L_p$ and a bounded sequence $\{y_n\}_{n=0}^{\infty}\subseteq\mathcal{M}$ such that $x_n=y_na$ for every $n\geq0$, and define a norm on $L_p(\mathcal{M},\ell_\infty^c)$ similarly.

The family $\mathcal{T}$ is called weak type $(p,p)$ for $1\leq p<\infty$ if there exists a constant $C>0$ such that, for every $x\in L_p^+$ and $\lambda>0$, there exists $e\in\mathcal{P}(\mathcal{M})$ such that
$$\tau(e^\perp)\leq\left(\frac{C\|x\|_p}{\lambda}\right)^p \ \text{ and } \sup_{n\in\mathbb{N}_0}\|eT^n(x)e\|_\infty\leq \lambda.$$
The family $\mathcal{T}$ will be said to be of strong type $(p,p)$ for $1\leq p\leq\infty$ if there exists $C>0$ such that, for every $x\in L_p$,
$$\|\{T^n(x)\}_{n=0}^{\infty}\|_{L_p(\mathcal{M};\ell_\infty)}\leq C\|x\|_p.$$

In \cite{jx}, it was shown that $\mathcal{T}$ being strong type $(p,p)$ is equivalent to: for every $x\in L_p^+$, there exists $a\in L_p^+$ and $C>0$ such that
$$\|a\|_p\leq C\|x\|_p \ \text{ and } \ T^n(x)\leq a\text{ for every }n\in\mathbb{N}_0.$$

It is known that $\mathcal{T}$ being strong type $(p,p)$ implies that it is weak type $(p,p)$ by using an appropriate spectral argument on $a$. Furthermore, one can observe that $\mathcal{T}$ being weak type $(p,p)$ implies that it is b.u.e.m. at zero on $(L_p,\|\cdot\|_p)$.

For technical reasons, adopting a non-standard notation, we will say that $\mathcal{T}$ is right-sided weak type $(p,p)$ if there exists a constant $C>0$ such that, for every $x\in L_p^+$ and $\lambda>0$, there exists $e\in\mathcal{P}(\mathcal{M})$ such that
$$\tau(e^\perp)\leq\left(\frac{C\|x\|_p}{\lambda}\right)^p \ \text{ and } \sup_{n\in\mathbb{N}_0}\|T^n(x)e\|_\infty\leq \lambda.$$
Similarly, we say that $\mathcal{T}$ is right-sided strong type $(p,p)$ if there exists $C>0$ such that, for every $x\in L_p$,
$$\|\{T^n(x)\}_{n=0}^{\infty}\|_{L_p(\mathcal{M};\ell_\infty^c)}\leq C\|x\|_p.$$ Similar relations between right-sided strong and weak type $(p,p)$ imply that $\mathcal{T}$ being either implies that it is u.e.m. at zero on $(L_p,\|\cdot\|_p)$.

\begin{ex}\label{e61}
As mentioned in the introduction, in the commutative setting (see \cite[Corollary 1]{st}), if $(X,\mathcal{F},\mu)$ is a probability space and $T:L_1(X,\mathcal{F},\mu)\to L_1(X,\mathcal{F},\mu)$ is a Dunford-Schwartz operator such that $\int_X (Tf)\overline{f}\ d\mu\geq0$ for every $f\in L_2(X,\mathcal{F},\mu)$ (in other words, the restriction of $T$ to $L_2(X,\mathcal{F},\mu)$ is a positive operator), then $\lim_{n\to\infty}T^nf(x)$ exists $\mu$-a.e. (and so $\{T^nf\}_{n=1}^{\infty}$ converges ($\mu$-)a.u. by Egorov's Theorem) for every $f\in L_p(X,\mathcal{F},\mu)$ whenever $1<p\leq\infty$.
\end{ex}

\begin{ex}\label{e62}
The previous example was generalized to the noncommutative setting in \cite[Theorem 6.7]{jx}, where it was shown that, if the restriction of $T$ to $L_2(\mathcal{M},\tau)$ is positive as a Hilbert space operator, then $T^n(x)\to F(x)$ b.a.u. as $n\to\infty$ for each $x\in L_p$ with $1<p<\infty$ and a.u. for $x\in L_p$ with $2<p<\infty$, where $F:L_p\to L_p$ is the projection of $L_p$ onto the fixed point space of $T$.

The family $\mathcal{T}$ is strong type $(q,q)$ for $1<q<\infty$ by \cite[Theorem 5.2]{jx}, and is right-sided strong type $(q,q)$ for $2<q<\infty$ by Corollary 5.9 of the same article.
\end{ex}

\begin{ex}\label{e63}
A Stoltz region with vertex $1$ is a set $D_\delta$, where $\delta>0$ and
$$D_\delta=\{z\in\mathbb{C}:|z|<1\text{ and }|1-z|<\delta(1-|z|)\}.$$

The numerical range of an operator $S$ on a Hilbert space $\mathcal{H}$ is the set
$$\Theta(S)=\{\langle S(\xi),\xi\rangle:\xi\in\mathcal{H},\|\xi\|=1\}.$$

Viewing $T\in DS^+$ as a contraction on $L_2$, if its numerical range is contained in a Stoltz region with vertex $1$, it was shown in \cite[Corollary 4.4]{be} that $T^n(x)\to F(x)$ b.a.u. (a.u.) as $n\to\infty$ for every $x\in L_p$ for $1<p<\infty$ (respectively, $2<p<\infty$). 

Furthermore, a strong type $(q,q)$ result is obtained for $\mathcal{T}$ when $1<q\leq\infty$, and a right-sided strong type $(q,q)$ result is obtained when $2<q\leq\infty$ by Theorem 3.4 and Corollary 3.5 of \cite{be}.
\end{ex}

\begin{ex}\label{e64}
One important example studied in \cite{be} that falls into the setting of Example \ref{e63} is as follows. Let $\mathcal{M}_1,...,\mathcal{M}_d$ be von Neumann subalgebras of $\mathcal{M}$ such that the restriction of $\tau$ to each $\mathcal{M}_k$ is semifinite. Let $\mathcal{E}_k$ be the normal faithful conditional expectation from $\mathcal{M}$ to $\mathcal{M}_k$ such that $\tau\circ\mathcal{E}_k=\tau$. Each $\mathcal{E}_k$ extends to a contractive projection from $L_p(\mathcal{M},\tau)$ to $L_p(\mathcal{M}_k,\tau)$. Then the operator $T:=\mathcal{E}_1\cdots\mathcal{E}_d$ is a positive Dunford-Schwartz operator which satisfies all of the conditions of Example \ref{e63}. It was noted in \cite{be} that the limit operator is easy to compute when the trace is finite - it is the $\tau$-preserving conditional expectation of $\mathcal{M}$ onto $\bigcap_{k=1}^{n}\mathcal{M}_k$.
\end{ex}

Each of the examples above guarantees that the b.a.u or a.u. convergence of the iterates $T^n(x)$ occurs for \textit{every} $x\in L_q$. However, for our results we only need this to hold whenever $x\in\mathcal{V}_q(T)$. The following results show that the assumptions on $T$ that were made throughout the paper can be obtained by showing that $T^k$ satisfies those assumptions for some $k\in\mathbb{N}$.

\begin{pro}\label{p61}
Fix $1<q<\infty$. Assume $T\in DS^+(\mathcal{M},\tau)$ is such that, for some integer $k\geq1$, the sequence $T^{km}(x)\to0$ b.a.u. (a.u.) as $m\to\infty$ for every $x\in\mathcal{V}_q(T^k)$. Then $T^n(x)\to0$ b.a.u. (respectively, a.u.) for every $x\in\mathcal{V}_q(T)$.
\end{pro}
\begin{proof}
We will prove the b.a.u. case only. Assume that $x\in\mathcal{V}_q(T)$. Since $T$ is a contraction on $L_q$, \cite[Corollary 9.18]{efhn} shows that $\mathcal{V}_q(T)$ is a subspace of $L_q$ which is invariant under $T$, so that $T^r(x)\in\mathcal{V}_q(T)$ for every $r\geq1$.

For the same reasons as above, we may apply \cite[Corollary 9.18]{efhn} again to find that $\mathcal{V}_q(T^k)=\mathcal{V}_q(T)$. So, by above and by assumption, we find that the sequence $\{T^{km}(T^r(x))\}_{m=0}^{\infty}$ converges to $0$ b.a.u. for every $r\in\{0,...,k-1\}$. Since $T^{km}(T^r(x))=T^{km+r}(x)$ for every $m\geq0$, this says that $T^{km+r}(x)\to0$ b.a.u. as $m\to\infty$ for every $r\in\{0,...,k-1\}$.

Assume $\epsilon>0$, and let $e_0,...,e_{k-1}\in\mathcal{P}(\mathcal{M})$ be such that $\tau(e_r^\perp)\leq\epsilon/k$ and $\|e_rT^{km+r}(x)e_r\|_\infty\to0$ as $m\to\infty$ for each $r\in\{0,...,k-1\}$. If we let $e=\bigwedge_{r=0}^{k-1}e_r$, then we find that $\tau(e^\perp)\leq\epsilon$ and $\|eT^{km+r}(x)e\|_\infty\to0$ as $m\to\infty$ for every $r\in\{0,...,k-1\}$.

Assume $\delta>0$. Let $M_r\in\mathbb{N}_0$ be such that $m\geq M_r$ implies $\|eT^{km+r}(x)e\|_\infty<\delta$. Let $M=\max\{M_0,...,M_{k-1}\}$. If $n=mk+r\geq (M+1)k$, where $m,r\in\mathbb{N}_0$ are such that $0\leq r<k$, then one finds that $m\geq M$, which in turn implies $\|eT^n(x)e\|_\infty=\|eT^{km+r}(x)e\|_\infty<\delta$. Therefore, since $n\geq(M+1)k$ and $\delta>0$ are arbitrary, it follows that $\|eT^n(x)e\|_\infty\to0$ as $n\to\infty$, and since $\epsilon>0$ is arbitrary, it follows that $T^n(x)\to0$ b.a.u. as $n\to\infty$.
\end{proof}

\begin{pro}\label{p62}
Fix $1\leq p<\infty$. Assume $T\in DS^+(\mathcal{M},\tau)$ is such that, for some integer $k\geq1$, the family $\{T^{km}\}_{m=0}^{\infty}$ is b.u.e.m. (u.e.m.) at $0$ on $(L_p,\|\cdot\|_p)$. Then $\{T^n\}_{n=0}^{\infty}$ is b.u.e.m. (respectively, u.e.m.) at $0$ on $(L_p,\|\cdot\|_p)$.
\end{pro}

\begin{proof}
We prove the b.u.e.m. case only.

Assume $\epsilon,\delta>0$. Let $\gamma>0$ be the value in the definition of $\{T^{km}\}_{m=0}^{\infty}$ being b.u.e.m. at zero on $(L_p,\|\cdot\|_p)$ corresponding to $\epsilon/k$ and $\delta$. 

Assume $x\in L_p$ satisfies $\|x\|_p<\gamma$. Since $T$ is a contraction on $L_p$, we have that $\|T^r(x)\|_p<\gamma$ for every $r\in\{0,...,k-1\}$. Therefore, for every such $r$, there exists $e_r\in\mathcal{P}(\mathcal{M})$ such that $\tau(e_r^\perp)\leq\epsilon/k$ and  $\sup_{m\in\mathbb{N}_0}\|e_rT^{km}(T^r(x))e_r\|_\infty\leq\delta.$ Let $e=\bigwedge_{r=0}^{k-1}e_r$, so that $$\tau(e^\perp)\leq\epsilon \, \text{ and } \, \sup_{m\in\mathbb{N}_0}\|eT^{km+r}(x)e\|_\infty\leq\delta \, \text{ for every } \,  r\in\{0,...,k-1\}.$$

If $n\in\mathbb{N}_0$, then there exists $m_n,r_n\in\mathbb{N}_0$ such that $n=m_nk+r_n$ with $0\leq r_n<k$, which implies $\|eT^n(x)e\|_\infty\leq\sup_{m\in\mathbb{N}_0}\|eT^{km+r_n}(x)e\|_\infty\leq\delta$. Since $n$ was arbitrary, we find that $\sup_{n\in\mathbb{N}_0}\|eT^n(x)e\|_\infty\leq\delta$, and since $\epsilon,\delta>0$ are arbitrary, it follows that $\{T^n\}_{n=0}^{\infty}$ is b.u.e.m. at zero on $(L_p,\|\cdot\|_p)$.
\end{proof}

Now, we will generalize Example \ref{e62}.

\begin{ex}\label{e65}
If the restriction of $T$ to $L_2(\mathcal{M},\tau)$ is self-adjoint, then $T^*=T$ implies that $T^2=T^*T$ is a positive operator on $L_2(\mathcal{M},\tau)$. Therefore, by Example \ref{e62} and Proposition \ref{p61}, $T^n(x)\to0$ b.a.u. (a.u) for every $x\in\mathcal{V}_q(T)$ with $1<q<\infty$ (respectively, $2<q<\infty$). The fact that $\{T^n\}_{n=0}^{\infty}$ is b.u.e.m. (respectively, u.e.m.) at zero on $(L_q,\|\cdot\|_q)$ can be obtained either by using both Example \ref{e62} and Proposition \ref{p62}, or directly through the results of \cite{jx}.

More generally, fix $n\in\mathbb{N}$ and let $U_n\subseteq\mathbb{T}$ denote the set of $n$-th roots of unity. Suppose that $T\in DS^+(\mathcal{M},\tau)$ is such that the restriction of $T$ to $L_2$ is normal and that $\sigma(T|_{L_2})$ is contained in $[0,1]U_n:=\{ru:r\in[0,1],u\in U_n\}$. Then $\sigma(T^n)\subseteq[0,1]$, and since $T^n$ is also normal, it follows that $T^n$ is a positive operator on $L_2$. Hence, we may apply Propositions \ref{p61} and \ref{p62} and Example \ref{e62} to find, for every $1<q<\infty$ ($2<q<\infty$), that $T^n(x)\to0$ b.a.u. (respectively, a.u.) for all $x\in\mathcal{V}_q(T)$ and that $\{T^n\}_{n=0}^{\infty}$ is b.u.e.m. (respectively, u.e.m.) at zero on $(L_q,\|\cdot\|_q)$.
\end{ex}

\begin{ex}\label{e69}
Let $\mathcal{M}$ be a von Neumann algebra with n.s.f. trace $\tau$ and let $\Phi:\mathcal{M}\to\mathcal{M}$ be a normal $\tau$-preserving $*$-automorphism (notice that it extends to a positive Dunford-Schwartz operator on $\mathcal{M}$). If $\mu$ is a probability measure on $\mathbb{Z}$, then the operator $\overline{\mu}$ defined by $$\overline{\mu}(x)=\sum_{n=-\infty}^{\infty}\mu(\{n\})\Phi^n(x), \ \text{ for } x\in L_1(\mathcal{M},\tau)+\mathcal{M},$$ is in $DS^+(\mathcal{M},\tau)$. If $\mu$ further satisfies the symmetry condition $\mu(\{-n\})=\mu(\{n\})$ for every $n\in\mathbb{Z}$, then the restriction of $\overline{\mu}$ to $L_2(\mathcal{M},\tau)$ is self-adjoint, and so falls into the setting of Example \ref{e65}.

In the commutative setting, a more general result of this nature is known for other assumptions on $\mu$, for which one can refer to \cite{jo} for details.

Let
$$\widehat{\mu}(\lambda)=\sum_{k\in\mathbb{Z}}\mu(\{k\})\lambda^{k}$$ denote the Fourier transform of $\mu$, where $\lambda\in\mathbb{T}$. Furthermore, given $\lambda\in\mathbb{T}$, assume that $|\widehat{\mu}(\lambda)|=1$ if and only if $\lambda=1$, and that
$$\sup_{|\lambda|=1;\lambda\neq1}\frac{|\widehat{\mu}(\lambda)-1|}{1-|\widehat{\mu}(\lambda)|}<\infty.$$

Let $(X,\mathcal{F},\nu,\phi)$ be an invertible probability measure preserving system, and consider the commutative von Neumann algebra $\mathcal{M}=L_\infty(X,\mathcal{F},\nu)$ acting via multiplication operators on $L_2(X,\mathcal{F},\nu)$, with $\tau$ being integration with respect to $\nu$ and $\Phi$ being the Koopman operator associated to $\phi$. Then, for every $1<p\leq\infty$ and $f\in L_p(X,\mathcal{F},\nu)$, the iterates $\{\overline{\mu}^n(f)\}_{n=0}^{\infty}$ converge $\nu$-a.e. by \cite[Theorem 4.4]{jo}, which implies a.u. convergence by Egorov's Theorem. That $\{\overline{\mu}^n\}_{n=0}^{\infty}$ is u.e.m. at zero on $(L_p(X,\mc{F},\nu),\|\cdot\|_p)$ for the same values of $p$ is a consequence of the a.u. convergence of $\{\overline{\mu}^n(f)\}_{n=0}^{\infty}$ for every $f\in L_p(X,\mc{F},\nu)$ as mentioned in Remark \ref{r21}.
\end{ex}

\begin{ex}\label{e66}

Consider $\mathbb{T}^d=\mathbb{R}^d/\mathbb{Z}^d$,
and let $m$ denote the normalized Lebesgue measure on $\mathbb{T}^d$. As in \cite{mcpo} we will construct a positive Dunford Schwartz operator on the noncommutative torus $\mathbb{T}_{\theta}^{d}$ defined below. Let $\theta=[\theta_{j,k}]_{j,k=1}^{d}$ be a real skew-symmetric $d\times d$-matrix. Define unitary operators $u_j:L_2(\mathbb{T}^d,m)\to L_2(\mathbb{T}^d,m)$ for $j=1,...,d$ by
$$(u_jf)(x_1,...,x_d)=e^{ix_j}f(x_1+\pi\theta_{j,1},x_2+\pi\theta_{j,2},...,x_d+\pi\theta_{j,d}).$$
The operators $u_1,...,u_d$ satisfy the following relations for $j,k=1,...,d$:
$$u_ku_j=e^{2\pi i\theta_{j,k}}u_ju_k$$
Let $\mathbb{T}_{\theta}^{d}$ be the von Neumann algebra on $L_2(\mathbb{T}^d,m)$ generated by $u_1,...,u_d$, and define a normal faithful tracial state $\tau$ on $\mathbb{T}_{\theta}^{d}$ by $$\tau(s)=\frac{1}{(2\pi)^d}\int_{\mathbb{T}^d}(s(\textbf{1}))(t)dm(t)=\langle s(\textbf{1}),\textbf{1}\rangle,$$ where $s\in\mathbb{T}_{\theta}^{d}$ and $\textbf{1}\in L_2(\mathbb{T}^d,m)$ denotes the constant function $\textbf{1}(t)=1$ for $t\in\mathbb{T}^d$.

Define operators $\delta_1,...,\delta_d:\mathbb{T}_{\theta}^{d}\to\mathbb{T}_{\theta}^{d}$ by
$$\delta_{j}(u_j)=2\pi iu_j \ \text{ and } \ \delta_j(u_k)=0 \text{ for }k\neq j,$$ and extend it linearly to the rest of the space.
As mentioned in \cite{cxy}, each operator $\delta_j$ acts like a partial derivative operator for functions on the usual $d$-torus. Letting $\Delta=\sum_{j=1}^{d}\delta_j^2$, one can show that $\Delta$ is a negative operator on $L_2(\mathbb{T}_{\theta}^{d},\tau)$ which is the generator of a semigroup of contractions on $L_2(\mathbb{T}_{\theta}^{d},\tau)$. Denoting this semigroup of operators by $(T_t)_{t\geq0}$, one obtains that $T_t=\exp(t\Delta)$, and that $T_t\in DS^+(\mathbb{T}_{\theta}^{d},\tau)$ for every $t\geq0$. The restriction of $T$ to $L_2(\mathbb{T}_{\theta}^{d},\tau)$ is a self-adjoint operator; in fact, each such operator is positive as an operator on $L_2(\mathbb{T}_{\theta}^{d},\tau)$ since $T_t=T_{t/2}T_{t/2}=T_{t/2}^*T_{t/2}$.

This is the heat semigroup of $\mathbb{T}_{\theta}^{d}$. In this setting, for any $t\in(0,\infty)$, the discretization $\{(T_t)^n\}_{n=0}^{\infty}$ satisfies the conditions mentioned in Example \ref{e62}.
\end{ex}

\begin{ex}\label{e67}
If $z\in\mathbb{T}$, then a Stoltz region with vertex $z$ is a set of the form $zD_\delta$ (so $x\in zD_\delta$ when $x=zy$ for $y\in D_\delta$), where $D_\delta$ is a Stoltz region with vertex $1$.

Suppose the restriction of $T$ to $L_2$ is a normal operator with $\sigma(T|_{L_2})\cap\mathbb{T}=\{\lambda_1,...,\lambda_d\}$, where $\lambda_1=1$. Assume further that, for every $z\in\sigma(T|_{L_2})\cap\mathbb{T}$, there exists a neighorhood $V_{z}$ of $z$ such that $\sigma(T)\cap V_{z}\setminus\{z\}$ is contained in a Stoltz region with vertex $z$. For each $k=1,...,d$, let $F_k$ denote the projection from $L_2$ onto $\ker(\lambda_k \operatorname{Id}_{L_2}-T)$. Then it was shown in \cite[Theorem 4.6]{be} that $T^n(x)$ converges b.a.u. to $F_1(x)$ for every $x\in\ker(F_2+...+F_d)$.

If $x\in\mathcal{V}_2(T)$, then $F_k(x)=0$ for each $k=1,2,...,d$. Therefore $x\in\ker(F_2+...+F_d)$, meaning $T^n(x)\to0$ b.a.u. for every $x\in\mathcal{V}_2(T)$. It was stated in \cite[Theorem 2.7]{be} that $\mathcal{T}$ is strong type $(2,2)$, and so is b.u.e.m. at zero on $(L_2,\|\cdot\|_2)$.

Unlike the other examples mentioned in this section, Theorems \ref{t43}, \ref{t51}, and \ref{t52} only apply for $T$ on $L_2(\mathcal{M},\tau)$. However, the conditions on $T$ are sufficient enough for the bilateral version of Theorem \ref{t42} to apply for $L_p(\mathcal{M},\tau)$ for every $1\leq p<\infty$.
\end{ex}

\begin{ex}\label{e68}
Let $\mathcal{M}$ be any semifinite von Neumann algebra, and let $u\in\mathcal{M}$ be normal contraction (as an operator on $\mathcal{H}$) whose spectrum, $\sigma(u)$, is a subset of $D_\delta\cup\{1\}$ for some $\delta>0$. Define $T_u:\mathcal{M}\to\mathcal{M}$ by $T_u(x)=u^*xu$. Then \cite[Lemma 1.1]{jx} applies to show that $T_u\in DS^+(\mathcal{M},\tau)$, and one sees that the restriction of $T_u$ to $L_2$ is normal. Furthermore, one finds that
$$\sigma(T_u)\subseteq\sigma(u^*)\sigma(u)\subseteq D_\delta\cup\{1\},$$ where $\sigma(u^*)\sigma(u)=\{ab:a\in\sigma(u^*),b\in\sigma(u)\}$. Therefore $T_u$ satisfies the conditions of Example \ref{e63}. One can also replace the condition that $\sigma(u)\subseteq D_\delta\cup\{1\}$ with $\sigma(u)\subseteq[0,1]$, $\sigma(u)\subseteq[-1,1]$, or $\sigma(u)\subseteq [0,1]U_n$, in which case the operator $T_u$ will satisfy the conditions of Example \ref{e62} or Example \ref{e65} instead. More details for this construction can be found in \cite[Example 2.8]{be}.
\end{ex}

\begin{ex}\label{e610}
Explicit examples of such an operator $T_u$ on $\mathcal{B}(\ell_2(\mathbb{N}_0))$ whose spectrum satisfy Example \ref{e68} are as follows. Let the sequence $\{\lambda_k\}_{k=0}^{\infty}$ be contained in one of $D_{\delta}\cup\{1\}$ or $[0,1]U_n$, where $\delta>0$ and $n\in\mathbb{N}_0$. Let $\{e_k\}_{k=1}^{\infty}$ denote the standard basis of $\ell_2(\mathbb{N})$. Defining $u\in\mathcal{B}(\ell_2(\mathbb{N}))$ by $u(e_k)=\lambda_k e_k$ for $k\in\mathbb{N}$ and extending linearly, one finds that $u\in\mathcal{B}(\ell_2(\mathbb{N}))$ and that $\sigma(u)$ is the closure of $\{\lambda_k:k\in\mathbb{N}\}$, which will be contained in whichever one of $D_\delta\cup\{1\}$ or $[0,1]U_n$ contained $\{\lambda_k:k\in\mathbb{N}\}$.
\end{ex}

\begin{rem}\label{r61}
The main results of this paper assume that the operator $T$ involved satisfies an assumption on its iterates. Although the examples provided above indicate that this assumption is validated, it would be desirable to weaken or find other similar assumptions for which analogous theorems can be obtained.

Also, many of the weighted ergodic theorems proved require the weights be in a certain subclass of Hartman sequences. Again, it would be desirable to know whether these results are valid for a bigger subclass of Hartman sequences. For example, in Section \ref{s5}, we showed that such results hold for $\Lambda$ even though $\Lambda\not\in W_{1^+}\cap H$. For what other sequences can we make similar assertions?

The subsequential ergodic theorems for sequences proved above (such as return time sequences, primes, and moving average sequences) are just a sample of subsequential results in this setting. Naturally, one would ask, and expect, if other subsequential ergodic theorems can be obtained in the same setting.
\end{rem}

\noindent\textbf{Acknowledgements.}
The author would like to thank Dr. Semyon Litvinov for his careful reading of the article and for his suggestions in improving numerous aspects of it.

\end{document}